\documentclass[twoside,leqno]{article}

\usepackage[letterpaper]{geometry}

\usepackage{siamproceedings}

\usepackage[T1]{fontenc}
\usepackage{amsfonts}
\usepackage{graphicx}
\usepackage{epstopdf}
\usepackage{enumitem}
\usepackage{algorithmic}
\ifpdf
  \DeclareGraphicsExtensions{.eps,.pdf,.png,.jpg}
\else
  \DeclareGraphicsExtensions{.eps}
\fi

\newsiamremark{remark}{Remark}
\newsiamremark{hypothesis}{Hypothesis}
\crefname{hypothesis}{Hypothesis}{Hypotheses}
\newsiamthm{example}{Example} 
\newsiamthm{claim}{Claim}
\newsiamthm{prob}{Problem} 

\usepackage{amsopn}

\def\mca{\mathcal}

\def\Q{{\mathbb Q}}
\def\Z{{\mathbb Z}}
\def\C{{\mathbb C}}
\def\R{{\mathbb R}}

\def\del{\partial}
\def\diff{\mathrm{Diff}}

\def\ech{\mathrm{ECH}}
\def\emb{\mathrm{emb}}
\def\ep{\varepsilon} 

\def\gen{\mathrm{gen}}

\def\Ham{\mathrm{Ham}} 

\def\id{\mathrm{id}} 

\def\image{\mathrm{Im}} 

\def\inj{\mathrm{inj}}

\def\MH{\mathrm{MH}}

\def\nondeg{\mathrm{nondeg}} 

\def\ph{\varphi} 
\def\pr{\mathrm{pr}}

\def\supp{\mathrm{supp}}

\def\vol{\mathrm{vol}}

\begin{document}

\title{\Large Strong closing lemmas in Hamiltonian dynamics} 
\author{Kei Irie  \thanks{Research Institute for Mathematical Sciences, Kyoto University (\email{iriek@kurims.kyoto-u.ac.jp})} } 

\date{\today}

\maketitle


\fancyfoot[R]{\scriptsize{Copyright \textcopyright\ 2026 by SIAM\\
Unauthorized reproduction of this article is prohibited}}





\begin{abstract}
This survey focuses on strong closing lemmas in Hamiltonian dynamics 
that are proved using spectral invariants (also known as action selectors) in symplectic geometry.  
We review strong closing lemmas in low-dimensional Hamiltonian dynamics
(Reeb flows on contact three-manifolds and area-preserving maps on symplectic surfaces) and outline the key ideas behind their proofs. 
We also discuss results concerning strong closing lemmas in high-dimensional Hamiltonian dynamics, 
as well as analogous results for minimal hypersurfaces. 
\end{abstract}

\section{Introduction.}

Periodic orbits (or closed orbits) are a fundamental object in the theory of dynamical systems.
In particular, periodic orbits of Hamiltonian systems can be characterized as critical points of the Hamiltonian action functional on the free loop space, which makes it possible to apply variational methods to detect them. This idea has been extensively developed in modern symplectic geometry.
For example, the Arnold conjecture on the number of periodic orbits of time-dependent Hamiltonian systems --- formulated as a far-reaching generalization of the classical Poincar\'{e}--Birkhoff theorem --- has been one of the central problems in the field.  

The \textit{closing lemma} offers a completely different approach to finding periodic orbits. 
This approach is also generally regarded as having its roots in the classical work of Poincar\'{e} \cite{Poincare}. 
Roughly speaking, the closing lemma asserts the following: 
given a dynamical system and a near-periodic orbit of the system, 
one can make this orbit genuinely periodic by a small perturbation of the system
--- in other words, one can ``close'' the given near-periodic orbit. 
The closing lemma has been an important topic in the theory of dynamical systems. 
In particular, Pugh's $C^1$-closing lemma (\cite{Pugh_closing_lemma}, \cite{Pugh_generic_density})is fundamental in the study of $C^1$-generic dynamics.
On the other hand, it has long been known that establishing $C^r$-closing lemmas for $r>1$ is extremely difficult. 

The first goal of this paper is to explain that, by employing techniques from modern symplectic geometry, we can establish very strong closing lemmas
---stronger than the $C^\infty$-closing lemmas---for certain low-dimensional Hamiltonian dynamics. After accomplishing this, we discuss possible extensions of these results to higher dimensions.
This is a survey article, with only a few new results in Section~7.1. 
Some material is adapted from another survey \cite{Irie_suugaku} by the author. In the rest of this section, we outline the contents of this paper.

\subsection{$C^1$-closing lemmas} 

Let us first formally state the $C^r$-closing problem for $r \in \Z_{\ge 1} \cup \{\infty\}$. 
Here we only consider dynamical systems defined by vector fields. 
Let $M$ be a closed manifold (i.e. $M$ is compact and $\del M= \emptyset$), 
and let $\mca{X}_{C^r}(M)$ denote the set of $C^r$-vector fields on $M$, 
equipped with the $C^r$-topology. 
This is a Bari\'{e} space, i.e. any residual subset (a subset which contains the intersection of countably many open and dense sets) is dense. 
For any $X \in \mca{X}_{C^r}(M)$, 
we define an isotopy $(\ph^t_X)_{t \in \R}$ on $M$ 
by $\ph^0_X = \id_M$ 
and $\del_t \ph^t_X = X(\ph^t_X)\,(\forall t \in \R)$. 
We say $p \in M$ is \textit{periodic}  if $\ph^t_X(p)=p$ for some $t>0$, 
and $p$ is \textit{nonwondering} 
if for any neighborhood $U$ of $p$ and any $T>0$, 
there exists $t>T$ such that $\ph^t_X(U) \cap U \ne \emptyset$. 
The set of all nonwondering points of $X$ is denoted by $\mathrm{NW}(X)$. 
Obviously, every periodic point is nonwandering. 
Although the converse implication is clearly false, it is natural to ask whether it holds after perturbing the vector field. 
This question is known as the closing problem. 
We now formally state the $C^r$-closing problem: 
\begin{quote} 
Let $p \in M$ be a nonwondering point of $X \in \mca{X}_{C^r}(M)$ 
and $\mca{U}$ be a neighborhood of $X$ in $\mca{X}_{C^r}(M)$, equipped with the $C^r$-topology. 
Does there exist $X' \in \mca{U}$ such that $p$ is a periodic point of $X'$? 
\end{quote} 
An affirmative answer to the $C^r$-closing problem is called the $C^r$-closing lemma. 
Obviously, the $C^r$-closing lemma gets stronger as $r$ gets larger. 
The $C^1$-closing lemma is already extremely hard to prove,  
and this was established by Pugh in his monumental work (\cite{Pugh_closing_lemma}, \cite{Pugh_generic_density}). 
As an important consequence of the $C^1$-closing lemma, Pugh \cite{Pugh_generic_density} proved the $C^1$-generic density of periodic points: 
for generic $X \in \mca{X}_{C^1}(M)$, the set of all periodic points of $X$ is dense in $\mathrm{NW}(X)$. 
Here `generic' means the following: for any topological space $T$, we say \textit{generic} $\tau \in T$ satisfies a certain property $P$ if the set of $\tau \in T$ satisfying $P$ is a residual subset of $T$. 

It is straightforward to formulate $C^r$-closing lemmas for more restrictive dynamical systems 
such as volume-preserving, symplectic, and Hamiltonian systems. 
Due to additional restrictions proving $C^1$-closing lemmas for these systems are more challenging, 
and this was achieved by Pugh--Robinson \cite{Pugh_Robinson}. 
In Section 2, we explain this result for Hamiltonian systems.

\subsection{$C^r$-closing lemmas for $r \ge 2$} 

It was already pointed out in \cite{Pugh_generic_density} that there is a fundamental difficulty to extend 
Pugh's proof of the $C^1$-closing lemma to prove the $C^r$-closing lemma for $r \ge 2$. 
Proving the $C^r$-closing lemma for $r \ge 2$ for general smooth dynamical systems 
is considered to be an extremely difficult problem, which is also in the Smale's 18 problems \cite{Smale}. 
On the other hand, ``Hamiltonian smooth closing lemma'' was disproved by Herman \cite{Herman}, as we discuss in Section 2. 

\subsection{Strong closing lemmas for low-dimensional Hamiltonian dynamics} 

In Section 3, we state strong closing lemmas 
for Reeb flows on contact $3$-manifolds and area-preserving maps of symplectic surfaces. 
The first result was proved by \cite{Irie_JMD}, and the second result was independently proved by 
\cite{CGPZ} and 
\cite{Edtmair_Hutchings} (combined with a result in \cite{CGPPZ}).  
Roughly speaking, strong closing lemma asserts that 
any $1$-parameter family of perturbations satisfying a very weak assumption can create a periodic orbit
which intersects the support of the perturbation. 
This is much stronger than the $C^\infty$-closing lemma, and in particular implies the $C^\infty$-generic density of periodic orbits. 

The proofs of these strong closing lemmas use the notion of \textit{spectral invariants} from symplectic geometry. 
In Section 4, we introduce this notion and explain a key observation (Lemma \ref{lem_local_sensitivity}) 
on how to use spectral invariants to prove strong closing lemmas. 
In Section 5, we mainly focus on the case of $3$-dimensional Reeb flows and outline two constructions of spectral invariants, both due to Hutchings, 
and explain properties of these invariants, including the ``Weyl law''---which is the key ingredient in the proofs of strong closing lemmas. 
The first construction uses ECH (embedded contact homology), and the second construction is an elementary alternative of the first one. 
Using the second construction, we sketch a relatively elementary proof of the strong closing lemma for 
the standard contact $3$-sphere (Section 5.4). 
In Section 6, we explain results on generic equidistribution of periodic orbits, which are quantitative refinements of the generic density of periodic orbits. 

\subsection{Strong closing lemmas for high-dimensional Hamiltonian dynamics} 

It is a wide open question whether the strong closing lemmas hold for high-dimensional Hamiltonian dynamics. 
In Section 7, we explain why an attempt to prove strong closing lemmas using elementary spectral invariants fails in high dimensions, 
and review recent results establishing strong closing lemmas for very special high-dimensional Hamiltonian systems. 

\subsection{Minimal hypersurfaces} 

In recent advances in the theory of minimal hypersurfaces, 
there is a story which is very similar to the one we have explained for low-dimensional Hamiltonian dynamics. 
In Section 8 we briefly explain this story. 

\section{Hamiltonian $C^1$-closing lemma and Herman's example} 

Hamiltonian system is a mathematical formulation of the equation of motion in classical mechanics. 
In the language of symplectic geometry, Hamiltonian system is a triple $(M,\omega, H)$, 
where $(M,\omega)$ is a $2n$-dimensional symplectic manifold\footnote{We always assume that manifolds are second-countable.} 
 (i.e. $\omega \in \Omega^2(M)$ is closed and nondegenerate), 
and $H$ is an $\R$-valued function on $M$. 
We assume that $H$ is at least $C^2$, 
and define the Hamiltonian vector field $X_H$ by the formula 
$i_{X_H}\omega= -dH$. 

Let us assume that $\del M = \emptyset$ and $H: M \to \R$ is proper. 
Then $X_H$ defines an isotopy $(\ph^t_{X_H})_{t \in \R}$, 
which we abbreviate by $(\ph^t_H)_{t \in \R}$. 
Cartan's formula shows that $L_{X_H}\omega=0$, 
thus each $\ph^t_H$ preserves $\omega$, 
in particular it preserves the volume form $\omega^n$. 
Then every $p \in M$ is nonwondering, since $H$ is proper. 

Let $\mca{H}_{C^2}(M)$ denote the set of $C^2$-proper functions on $M$. 
Equipped with the $C^2$-Whitney topology, $\mca{H}_{C^2}(M)$ is a Bair\'{e} space (see \cite{Pugh_Robinson}).  
Now we can state the Hamiltonian $C^1$-closing lemma by Pugh--Robinson \cite{Pugh_Robinson}: 

\begin{theorem}[\cite{Pugh_Robinson}] 
Let $H \in \mca{H}_{C^2}(M)$ and $p \in M$. 
For any neighborhood $\mca{U}$ of $H$ in $\mca{H}_{C^2}(M)$, 
there exists $H' \in \mca{U}$ such that $p$ is a periodic point of $X_{H'}$. 
\end{theorem}

Note that if $H'$ is $C^2$-close to $H$ then $X_{H'}$ is $C^1$-close to $X_H$. 
This is why we call this theorem $C^1$-closing lemma. 
An important corollary of the Hamiltonian $C^1$-closing lemma is the $C^1$-generic density of periodic orbits: 

\begin{corollary}\label{cor_Hamiltonian_generic_density} 
For generic $H \in \mca{H}_{C^2}(M)$, the set of periodic points of $X_H$ is dense in $M$. 
\end{corollary} 
\begin{proof}[Proof (sketch)]
A nonconstant periodic orbit of $X_H$ is called \textit{nondegnerate} 
if $1$ is not an eigenvalue of its linearized return map. 
For any nonempty open set $U \subset M$, let $\mca{H}_U$ denote the set of $H \in \mca{H}_{C^2}(M)$ such that 
there exists a nondegenerate periodic orbit of $X_H$ which intersects $U$. 
It is easy to see that $\mca{H}_U$ is open, and 
the Hamiltonian $C^1$-closing lemma shows that $\mca{H}_U$ is dense.
Take a countable base of open sets of $M$, which we denote by $(U_i)_i$. 
Then $\bigcap_i \mca{H}_{U_i}$ is residual. 
If $H \in \bigcap_i \mca{H}_{U_i}$ then the set of periodic points of $X_H$ is dense in $M$. 
\end{proof} 

\begin{remark} 
Proving the $C^1$-generic density of periodic orbits for non-conservative dynamics 
(which was achieved in \cite{Pugh_generic_density}) is harder, since $\mathrm{NW}(X)$ may vary on $X$. 
\end{remark} 

As we mentioned in the introduction, 
``Hamiltonian smooth closing lemma'' was disproved by Herman \cite{Herman}. 

\begin{theorem}[\cite{Herman}] 
For any $n \in \Z_{\ge 2}$, 
there exists a symplectic form $\omega$ on $T^{2n}:= (\R/\Z)^{2n}$
and nonempty open sets  $U \subset T^{2n}$ and $\mca{U} \subset C^{2n+2}(T^{2n})$ 
such that for any $H \in \mca{U}$, no periodic orbit of $X_H$ intersects $U$. 
\end{theorem} 

The proof of the non-existence of periodic orbits is based on a version of KAM theory; see \cite{HZ} Section 4.5 for details. 
The author is not aware whether the exponent $2n+2$ can be improved or not.

\section{Strong closing lemmas in low-dimensional Hamiltonian dynamics} 

We recall basics of contact manifolds and Reeb flows (Section 3.1), 
state the strong closing lemma for $3$-dimensional Reeb flows (Section 3.2), 
discuss its consequence for closed geodesics on surfaces (Section 3.3), 
and state the strong closing lemma for area-preserving surface maps (Section 3.4). 

\subsection{Contact manifolds and Reeb flows} 

Let $n \in \Z_{\ge 1}$ and 
let $Y$ be a $2n-1$-dimensional manifold. 
$\lambda \in \Omega^1(Y)$ is called a \textit{contact form} 
if $\lambda \wedge (d\lambda)^{n-1}(p) \ne 0$ for all $p \in Y$. 
For any contact form $\lambda$, 
the hyperplane field $\xi_\lambda:= \ker(\lambda)$ on $Y$ 
has a natural co-orientation: 
$[v] \in TY/\xi_\lambda$ is positive if and only if $\lambda(v)>0$. 

A co-oriented hyperplane field $\xi$ on $Y$ is called a \textit{contact distribution} 
if there exists a contact form $\lambda$ such that $\xi_\lambda=\xi$. 
The pair $(Y,\xi)$ is called a \textit{contact manifold}, 
and $\Lambda(Y, \xi)$ denotes the set of contact forms $\lambda$ such that $\xi_\lambda=\xi$. 
It is easy to see that $\Lambda(Y,\xi)=\{ e^h \lambda \mid h \in C^\infty(Y) \}$ for any $\lambda \in \Lambda(Y,\xi)$. 

For any contact form $\lambda$, 
the \textit{Reeb vector field}  $R_\lambda \in \mathfrak{X}(Y)$
is characterized by the equations 
$i_{R_\lambda} (d\lambda) \equiv 0$ and $\lambda(R_\lambda) \equiv 1$. 
A \textit{periodic Reeb orbit} of $\lambda$ 
is a map $\gamma: S^1 \to Y$ such that 
there exists $T_\gamma>0$ satisfying $\dot{\gamma} = T_\gamma \cdot R_\lambda(\gamma)$. 
We call $T_\gamma$ the \textit{period} or \textit{action} of $\gamma$. 
The set of all periodic Reeb orbits of $\lambda$ is denoted by $\mca{P}(Y,\lambda)$. 
The \textit{generalized Weinstein conjecture} states that
$\mca{P}(Y,\lambda) \ne \emptyset$
for any contact form $\lambda$ on any closed manifold $Y$. 
Taubes \cite{Taubes_Weinstein} fully proved the conjecture for the case $\dim Y = 3$.
The conjecture is still widely open for the case $\dim Y \ge 5$. 

$\gamma \in \mca{P}(Y, \lambda)$ is called \textit{simple} if it is injective as a map from $S^1$ to $Y$. 
Let $\mca{P}_\emb(Y,\lambda)$ denote 
the quotient of the set of all simple periodic Reeb orbits of $\lambda$ 
divided by the natural $S^1$-action on it. 

$\gamma \in \mca{P}(Y,\lambda)$ is called \textit{nondegenerate} if $1$ is not an eigenvalue of its linearized return map. 
$\lambda$ is called \textit{nondegenerate} if all its periodic Reeb orbits (including iterated orbits) are nondegenerate. 
For any contact manifold $(Y, \xi)$, let $\Lambda_\nondeg(Y,\xi)$ denote the set of nondegenerate elements in $\Lambda(Y,\xi)$. 
It is known that $\Lambda_\nondeg(Y,\xi)$ is a residual subset of $\Lambda(Y,\xi)$ with the $C^\infty$-topology. 
Let us see two basic examples of Reeb flows. 

\begin{example}[Boundary of symplectic ellipsoid]\label{ex_ellipsoids} 
For any $n \in \Z_{\ge 1}$, let $A_n$ be the set of tuples 
$a=(a_1,\ldots, a_n) \in \R^n$ such that $0<a_1 \le \cdots \le a_n$. 
For any $a \in A_n$, let 
\[ 
E_a:= \bigg\{ (q_1, \ldots, q_n, p_1, \ldots, p_n) \in \R^{2n} \biggm{|}  \sum_{j=1}^n  \frac{ \pi (q_j^2+p_j^2)}{a_j} \le 1  \bigg\}.
\] 
Let $\lambda_n:= \frac{1}{2} \sum_{j=1}^n p_j dq_j - q_j dp_j \in \Omega^1(\R^{2n})$. 
For each $a \in A_n$, $\lambda_a:= \lambda_n|_{\del E_a}$ is a contact form on $\del E_a$, 
and $R_{\lambda_a}=\sum_{j=1}^n (2\pi/a_j) \cdot (p_j \del_{q_j} - q_j \del_{p_j})$. 
If $a_i/a_j \in \Q$ for all $i \le j$, 
then all Reeb orbits are periodic. 
If $a_i/a_j \not\in \Q$ for all $i<j$, 
then $\# \mca{P}_\emb(\del E_a, \lambda_a)=n$.
\end{example} 

\begin{example}[Geodesic flow]\label{ex_geodesic} 
For any $C^\infty$-manifold $N$, 
let us define $\lambda_N \in \Omega^1(T^*N)$ by 
\[ 
\lambda_N(v):= p (\pr_*(v)) \qquad( q \in N, \, p \in T_q^*N, \, v \in T_{(q,p)}(T^*N)). 
\] 
Here $\pr: T^*N \to N$ is the projection to $N$. 
For any Finsler (in particular, Riemannian) metric $g$ on $N$, 
consider the unit sphere cotangent bundle 
$S^*_gN:= \{ (q,p) \in T^*N \mid    q \in N, \, p \in T_q^*N, \,  \|p \|_g = 1\}$. 
Then $\lambda_N|_{S_g^*N}$ is a contact form, and its Reeb vector field generates the geodesic flow. 
Periodic Reeb orbits of this contact form correspond to nonconstant closed geodesics of $g$. 
\end{example}

For any pair $(Y,\lambda)$ of a manifold $Y$ and a contact form $\lambda$, 
let $X:= Y \times \R$ 
and $\omega:=d(e^r\lambda) \in \Omega^2(X)$, 
where $r$ denotes the coordinate on $\R$. 
This is called the symplectization of the pair $(Y,\lambda)$. 
Setting $H: X \to \R$ by $H(y,r):=e^r$, we obtain $X_H=R_\lambda$. 
Hence Reeb flows are Hamiltonian systems. 
In particular, one can deduce the $C^1$-closing lemma for Reeb flows 
from the Hamiltonian $C^1$-closing lemma: 

\begin{theorem} 
Let $\lambda$ be a contact form on a closed manifold $Y$. 
For generic $h \in C^2(Y, \lambda)$, 
the union of periodic Reeb orbits of $e^h\lambda$ is dense in $Y$. 
\end{theorem} 

As we explained, there exists a Hamiltonian system on a $4$-dimensional symplectic manifold 
which violates the ``Hamiltonian $C^5$-closing lemma''. 
In contrast, as we explain in the next subsection, 
the strong closing lemma --- which is stronger than the $C^\infty$-closing lemma---
holds for Reeb flows on contact $3$-manifolds. 

\subsection{Reeb flows on contact $3$-manifolds} 

Let us first introduce the notion of strong closing property. 

\begin{definition} 
Let $\lambda$ be a contact form on a closed manifold $Y$. 
We say $\lambda$ satisfies \textit{strong closing property}
if the following holds: 
\begin{quote} 
For any $h \in C^\infty(Y, \R_{\ge 0}) \setminus \{0\}$, there exist
$t \in [0,1]$ and $\gamma \in \mca{P}(Y, e^{th}\lambda)$ such that 
$\image(\gamma) \cap \supp (h) \ne \emptyset$.
\end{quote} 
\end{definition} 

A few remarks on this definition are in order. 
\begin{itemize} 
\item If $\bigcup_{\gamma \in \mca{P}(Y,\lambda)} \image(\gamma)$ is dense in $Y$, 
in particular if all Reeb orbits are periodic, then $\lambda$ satisfies strong closing property. 
\item If $\lambda$ satisfies strong closing property, then $\lambda$ satisfies the $C^\infty$-closing property. 
Namely, for any nonempty open set $U \subset Y$ and a neighborhood $\mca{U}$ of $0 \in C^\infty(Y)$, 
there exist $f \in \mca{U}$ and $\gamma \in \mca{P}(Y, e^f\lambda)$ such that $\image(\gamma) \cap U \ne \emptyset$. 
Indeed, we can apply strong closing property for $h \in C^\infty(Y, \R_{\ge 0}) \setminus \{0\}$ 
such that $\supp(h) \subset U$ and $th \in \mca{U}$ for all $t \in [0,1]$. 
\end{itemize} 

It would be clear from the above argument that strong closing property is much stronger than the $C^\infty$-closing property. 
However, the following result holds true: 

\begin{theorem}[Strong closing lemma for $3$-dimensional Reeb flows \cite{Irie_JMD}]\label{thm_strong_closing_dim3} 
Any contact form $\lambda$ on any closed $3$-manifold $Y$ satisfies strong closing property.
\end{theorem} 

As a corollary, we obtain the generic density of periodic orbits. 
The proof is parallel to that of Corollary \ref{cor_Hamiltonian_generic_density}. 

\begin{corollary}[\cite{Irie_JMD}] 
Let $(Y, \lambda)$ be as in Theorem \ref{thm_strong_closing_dim3}. 
For generic $h \in C^\infty(Y)$, 
the union of periodic Reeb orbits of $e^h \lambda$ is dense in $Y$. 
\end{corollary} 

The proof in \cite{Irie_JMD} uses properties of ECH (embedded contact homology) spectral invariants, 
in particular the ``Weyl law'' which was proved by Crisotfaro-Gardiner--Hutchings--Ramos \cite{CGHR}. 
We will explain the key observation in \cite{Irie_JMD} (which is independent of the construction of ECH) in Section 4.1, 
and discuss ECH spectral invariants in Section 5.1. 

\subsection{Closed geodesics on surfaces} 

Since geodesic flows are Reeb flows (Example \ref{ex_geodesic}), 
Theorem \ref{thm_strong_closing_dim3} has the following corollary: 

\begin{corollary}[\cite{Irie_JMD}]\label{cor_geodesics} 
Let $(M,g)$ be a $2$-dimensional closed Riemannian manifold. 
For any $h \in C^\infty(M, \R_{\ge 0}) \setminus \{0\}$, 
there exist $t \in [0,1]$ and a nonconstant closed geodesic $\gamma$ of $e^{th}g$ 
which intersects $\supp(h)$. 
\end{corollary}

As a consequence, we obtain generic density of closed geodesics: 

\begin{corollary}\label{cor_geodesics_dense}
Let $(M,g)$ be a $2$-dimensional closed Riemannian manifold. 
For generic $h \in C^\infty(M)$, the union of all nonconstant closed geodesics of $e^h g$ is dense in $M$. 
\end{corollary} 

Note that we cannot regard Corollary \ref{cor_geodesics} as a closing lemma for geodesic flows,
because it controls only the projections of periodic orbits of the geodesic flow to $M$.
Rifford \cite{Rifford} proved a closing lemma for geodesic flows under $C^1$-perturbations of Riemannian metrics.
As far as the author knows, it is not known whether the $C^2$-version of this result holds.
We cannot (at least directly) apply the Hamiltonian $C^1$-closing lemma,
since the set of kinetic Hamiltonians of Riemannian metrics is not open in the space of $C^2$-Hamiltonians.

\subsection{Area-preserving maps of surfaces} 

Let us first recall some basic notions about symplectic manifolds of arbitrary dimensions. 
Let $(M, \omega)$ be a closed symplectic manifold. 
$\diff(M, \omega)$ denotes the set of symplectic diffeomorphisms of $(M, \omega)$, 
i.e. diffeomorphisms preserving $\omega$. 

Let $\mca{H}(M)$ denote the set of $H \in C^\infty([0,1] \times M)$ 
such that $\supp (H) \subset (0,1) \times M$. 
For any $H \in \mca{H}(M)$, 
we define the Hamiltonian isotopy $(\ph^t_H)_{t \in [0,1]}$ by 
$\ph^0_H=\id_M$ and $\del_t \ph^t_H = X_{H_t}(\ph^t_H)\,(\forall t \in [0,1])$, 
where $H_t \in C^\infty(M)$ is defined by $H_t(p):= H(t,p)$. 
Each map $\ph^t_H$ preserves $\omega$. 
The elements of 
$\Ham(M, \omega):= \{ \ph^1_H \mid H \in \mca{H}(M) \}$ 
are called Hamiltonian diffeomorphisms of $(M, \omega)$.

For the next definition we need a few more notations. 
For any $H \in \mca{H}(M)$, we abbreviate $\mathrm{pr}_M(\supp(H))$ as $\supp(H)$. 
For any $\psi \in \diff(M,\omega)$, we define $\psi_H:= \psi \circ \ph^1_H \in \diff(M,\omega)$. 
For any $\psi \in \diff(M)$, $\mca{P}(\psi)$ denotes the set of periodic points of $\psi$. 
Finally, let $\mca{H}_+(M):= \mca{H}(M) \cap C^\infty([0,1] \times M, \R_{\ge 0})$. 

\begin{definition} 
$\psi \in \diff(M, \omega)$ satisfies strong closing property 
if for any $H \in \mca{H}_+(M) \setminus \{0\}$ 
there exists $\tau \in [0,1]$ such that 
$\mca{P}(\psi_{\tau H}) \cap \supp(H)  \ne \emptyset$. 
\end{definition} 

For any $\psi \in \diff(M, \omega)$, we consider the mapping torus 
$M_\psi:= [0,1] \times M/ \sim$, 
where $\sim$ is defined by 
$(1,p) \sim (0, \psi(p)) \, (\forall p \in M)$. 
Since $\psi$ preserves $\omega$, 
it is easy to see that 
$(\pr_M)^*\omega \in \Omega^2([0,1] \times M)$ 
descends to a $2$-form on $M_\psi$, 
which we denote by $\omega_\psi$. 
$\psi$ is called \textit{rational} 
if $[\omega_\psi] \in H^2(M_\psi:\R)$ 
is proportional to an element of $H^2(M_\psi:\Z)$. 
For later purposes we introduce a few more notations. 
Let $t$ denote the coordinate on $[0,1]$. 
Then $\del_t \in \mca{X}([0,1]\times M)$ descends to a vector field on $M_\psi$,
which we denote by $R_\psi$. 
An \textit{orbit set} of $\psi$ is a $1$-cycle in $M_\psi$ of the form $\sum_i a_i \gamma_i$, 
where $a_i \in \Z_{>0}$ and $\gamma_i$ is an embedded periodic orbit of $R_\psi$.

Now we can state the strong closing lemma for rational area-presreving maps: 

\begin{theorem}[\cite{CGPZ}, \cite{CGPPZ}, \cite{Edtmair_Hutchings}]\label{thm_closing_lemma_for_surfaces} 
Let $(\Sigma, \omega)$ be a $2$-dimensional closed symplectic manifold. 
If $\psi \in \diff(\Sigma, \omega)$ is rational, then $\psi$ satisfies strong closing property. 
\end{theorem} 

This theorem was independently proved by Cristofaro-Gardiner--Prasad--Zhang \cite{CGPZ} and by 
Edtmair--Hutchings \cite{Edtmair_Hutchings} (combined with the result by Cristofaro-Gardiner--Pomerleano--Prasad--Zhang \cite{CGPPZ}), 
although the theorem as stated above is implicit in both papers. 
\cite{Edtmair_Hutchings} proved the theorem under an additional assumption that $\psi$ satisfies a property called \textit{$U$-cycle property}, 
and \cite{CGPPZ} proved that all rational area-preserving maps satisfy this property, thus it turned out that the assumption is redundant. 
The proof in \cite{CGPZ} requires that $\psi$ is nondegenerate, however this assumption can be easily dropped. 

Both proofs use spectral invariants of PFH (periodic Floer homology), which is an analogue of ECH for (mapping tori of) area-preserving surface maps. 
For Hamiltonian diffeomorphisms of surfaces,  Edtmair \cite{Edtmair_elementary} gave another proof using elementary alternative of PFH spectral invariants. 
The assumption that $\psi$ is rational cannot be dropped. 
For example, a Diophantine rotation on $T^2$ fails to satisfy the $C^\infty$-closing property (Herman \cite{Herman_torus}). 
This rationality assumption is needed to ensure that PFH spectral invariants take values in a fixed closed null set (see Section 4.2). 
As in the case of Reeb flows, Theorem \ref{thm_closing_lemma_for_surfaces} implies the generic density of periodic orbits: 

\begin{corollary}\label{cor_generic_density_surface} 
Let $(\Sigma, \omega)$ be a $2$-dimensional closed symplectic manifold. 
Then, $\mca{P}(\psi)$ is dense in $\Sigma$ for 
(i) $C^\infty$-generic $\psi \in \Ham(\Sigma, \omega)$ and 
(ii) $C^\infty$-generic $\psi \in \diff(\Sigma, \omega)$. 
\end{corollary}
\begin{proof} 
(i) follows because any Hamiltonian diffeomorphism of $(\Sigma, \omega)$ is rational. 
(ii) follows because the set of rational area-preserving maps is $C^\infty$-dense in $\diff(\Sigma, \omega)$ (see \cite{CGPZ} Lemma 5.4). 
\end{proof} 

Corollary \ref{cor_generic_density_surface} (i) was already proved in \cite{Asaoka_Irie}
by reducing it to Theorem \ref{thm_strong_closing_dim3}, 
but the strong closing property was not established in that paper.

\section{Spectral invariants and strong closing lemmas} 

We introduce the notion of spectral invariants 
and explain the key observation (Lemma \ref{lem_local_sensitivity}) to prove strong closing lemmas. 
We mainly discuss the case of Reeb flows (Section 4.1), 
and briefly mention the case of (rational) symplectic diffeomorphisms (Section 4.2). 

\subsection{Reeb flows on contact manifolds} 

Let $(Y,\xi)$ be a closed contact manifold of dimension $2n-1$. 
For any $\lambda \in \Lambda(Y,\xi)$, we consider the following sets of real numbers: 
\begin{align*} 
\mca{A}(Y,\lambda)&:= \{ T_\gamma \mid \gamma \in \mca{P}(Y,\lambda) \}, \\ 
\mca{A}_+(Y,\lambda)&: =\{0\} \cup  \{ a_1+ \cdots + a_m \mid m \ge 1, \, a_1, \ldots, a_m \in \mca{A}(Y,\lambda)\}. 
\end{align*} 
 We call the set $\mca{A}_+(Y,\lambda)$ generalized action spectrum of $(Y,\lambda)$. 
It is easy to see that $\mca{A}(Y,\lambda)$ is a closed subset of $\R_{>0}$ and $\min \mca{A}(Y,\lambda)>0$,
thus $\mca{A}_+(Y, \lambda)$ is a closed subset of $\R_{\ge 0}$. 
Moreover, the following fact plays a key role in our argument. 
The proof is an elementary application of  (finite-dimensional) Sard's theorem (see \cite{Irie_JMD} Lemma 2.2). 

\begin{lemma}\label{lem_nullset} 
$\mca{A}_+(Y,\lambda)$ is a null set (i.e. Lebesgue measure zero set). 
\end{lemma}

Let us introduce the notion of spectral invariants in an axiomatic way. 

\begin{definition}\label{defn_action_selector} 
Let $(Y,\xi)$ be a closed contact manifold. 
A spectral invariant of $(Y,\xi)$ is a map 
$s: \Lambda(Y, \xi) \to \R_{\ge 0}$
such that the following properties hold for any $\lambda \in \Lambda(Y, \xi)$: 
\begin{description} 
\item[(Spectrality)] $s(\lambda) \in \mca{A}_+(Y, \lambda)$. 
\item[(Conformality)] $s(a\lambda) = a \cdot s(\lambda)$ for any $a \in \R_{>0}$. 
\item[(Monotonicity)] $s(e^h\lambda) \ge  s(\lambda)$ for any $h \in C^\infty(Y, \R_{\ge 0})$. 
\item[($C^0$-continuity)] For any $\ep \in \R_{>0}$ there exists $\delta \in \R_{>0}$ such that 
$|s(\lambda) - s(e^h\lambda)| \le \ep$ for any $h \in C^\infty(Y, [-\delta, \delta])$. 
\end{description} 
\end{definition} 

A few remarks on this definition are in order: 
\begin{itemize} 
\item  $C^0$-continuity easily follows from conformality and monotonicity.
\item If $s: \Lambda_\nondeg(Y, \xi) \to \R_{\ge 0}$ satisfies the above properties, 
then it uniquely extends to a spectral invariant $\bar{s}: \Lambda(Y,\xi) \to \R_{\ge 0}$. 
\end{itemize}

The term spectral invariant (also known as action selector) does not denote a single invariant, 
but rather refers to a class of invariants 
that 
assign a critical value of a Hamiltonian action functional 
to 
an object in symplectic geometry, 
in a way several appropriate conditions are satisfied. 
Spectral invariants can also be defined for Hamiltonian or symplectic diffeomorphisms, as well as for pairs of Lagrangian submanifolds. 
Spectral invariants are a key concept in modern symplectic geometry, 
and are frequently used to establish various symplectic rigidity phenomena. 
\cite{HZ} is an excellent introductory text for this circle of ideas. 

Now we are ready to explain our key observation: Lemma \ref{lem_local_sensitivity}. 
We say $\lambda \in \Lambda(Y, \xi)$ satisfies \textit{local sensitivity} 
if for any $h \in C^\infty(Y, \R_{\ge 0}) \setminus \{0\}$ 
there exists a spectral invariant $s$ such that $s(e^h\lambda)>s(\lambda)$. 

\begin{lemma}\label{lem_local_sensitivity} 
If $\lambda \in \Lambda(Y, \xi)$ satisfies the local sensitivity then it satisfies the strong closing property. 
\end{lemma} 
\begin{proof} 
Suppose that $\lambda$ does not satisfy strong closing property. 
Then there exists $h \in C^\infty(Y, \R_{\ge 0}) \setminus \{0\}$ such that the following holds: 
\begin{equation}\label{eqn_local_sensitivity} 
t \in [0,1],\,  \gamma \in \mca{P}(Y, e^{th} \lambda) \implies \image(\gamma) \cap \supp(h) = \emptyset. 
\end{equation}
Since $e^{th}\lambda$ coincides $\lambda$ on $Y \setminus \supp(h)$, 
(\ref{eqn_local_sensitivity}) implies 
$\mca{P}(Y, e^{th}\lambda) = \mca{P}(Y, \lambda)$, 
thus $\mca{A}_+(Y, e^{th}\lambda) = \mca{A}_+(Y, \lambda)$ for any $t \in [0,1]$. 
For any spectral invariant $s$, the spectrality implies that 
$s(e^{th}\lambda) \in \mca{A}_+(Y, e^{th}\lambda) = \mca{A}_+(Y,\lambda)$. 
By the continuity of spectral invariants, 
the function $s(e^{th}\lambda)$ is continuous on $t$. 
On the other hand $\mca{A}_+(Y,\lambda)$ is a null set (Lemma \ref{lem_nullset}), 
thus $s(e^{th}\lambda)$ is constant on $t$.
In particular 
$s(\lambda) = s(e^h\lambda)$ for any spectral invariant $s$, 
thus $\lambda$ does not satisfy the local sensitivity. 
\end{proof} 

By Lemma \ref{lem_local_sensitivity}, 
Theorem \ref{thm_strong_closing_dim3} reduces to the following result: 

\begin{theorem}\label{thm_local_sensitivity_dim3} 
Any contact form on any closed contact $3$-manifold satisfies the local sensitivity. 
\end{theorem} 

As we explain in Section 5, 
Theorem \ref{thm_local_sensitivity_dim3} 
is a consequence of the ``Weyl law'' 
of ECH spectral invariants. 
We can also use an elementary alternative of ECH spectral invariants, 
which gives a better understanding on why 
Theorem \ref{thm_local_sensitivity_dim3} is a special property in dimension $3$.

\subsection{Rational symplectic diffeomorphisms} 

It is not difficult to state a result similar to Lemma \ref{lem_local_sensitivity} 
for rational symplectic diffeomorphisms. 
Let $(M,\omega)$ be a closed symplectic manifold, 
and $\psi \in \diff(M, \omega)$ be a rational symplectic diffeomorphism. 
For any $H \in \mca{H}(M)$, 
we define a diffeomorphism
\[ 
f_H: M_\psi \to M_{\psi_H}; \,  [ (t,p)] \mapsto [(t, (\ph^t_H)^{-1}(p)] \qquad(t \in [0,1] , \, p \in M). 
\] 
For any $\sigma \in H_1(M_\psi:\Z)$, 
let us take a $\Z$-coefficient cycle $\gamma_\sigma$ in $M_\psi$ 
representing $\sigma$, which we call a reference cycle. 
Then for any $H \in \mca{H}(M)$, 
let $\mca{Z}_\sigma(H)$ denote the set of pairs $(Z,\alpha)$, 
where $\alpha$ is an orbit set of $\psi_H$, 
and $Z$ is a $\Z$-coefficient $2$-chain in $M_\psi$ 
such that $\del Z  = \alpha - (f_H)_*(\gamma_\sigma)$. 
Let us define $\mca{A}_\sigma(H) \subset \R$ by 
\[ 
\mca{A}_\sigma(H): = \bigg\{ \int_Z \omega_\psi \biggm{|}  (Z,\alpha) \in \mca{Z}_\sigma(H)  \bigg\}. 
\] 
The assumption that $\psi$ is rational implies that $\mca{A}_\sigma(H)$ is a closed subset of $\R$. 
Also, an argument similar to the proof of Lemma \ref{lem_nullset} shows that $\mca{A}_\sigma(H)$ is a nullset. 

Let us define the notion of spectral invariants of $\psi$. 
A map 
\[ 
s: H_1(M_\psi: \Z) \times \mca{H}(M) \to \R; \, (\sigma, H) \mapsto s_\sigma(H) 
\] 
is called a spectral invariant if the following properties are satisfied: 
\begin{description} 
\item[(Spectrality)]$s_\sigma(H) \in \mca{A}_\sigma(H)$ for any $\sigma \in H_1(M_\psi:\Z)$ and $H \in \mca{H}(M)$. 
\item[(Continuity)]  For each $\sigma \in H_1(M_\psi:\Z)$, the map $s_\sigma$ is continuous with respect to the $C^0$-norm on $\mca{H}(M)$. 
\end{description} 

Now we can state an analogue of Lemma \ref{lem_local_sensitivity}. 

\begin{lemma}\label{lem_local_sensitivity_diffeo} 
Suppose that for any $H \in \mca{H}_+(M) \setminus \{0\}$ 
there exist a spectral invariant $s$ and $\sigma \in H_1(M_\psi: \Z)$ 
such that $s_\sigma(H) \ne s_\sigma(0)$. 
Then $\psi$ satisfies the strong closing property. 
\end{lemma} 

The proof of this lemma is parallel to that of Lemma \ref{lem_local_sensitivity}. 
As we explain in Section 5.5, 
for rational area preserving surface maps, 
the assumption of this lemma is satisfied 
with spectral invariants defined from PFH (periodic Floer homology).

\section{Constructions of spectral invariants} 

We briefly explain the construction of ECH spectral invariants (Section 5.1), 
state the Weyl law, and deduce Theorem \ref{thm_local_sensitivity_dim3} (Section 5.2). 
Then we explain the construction of an elementary alternative to ECH spectral invariants (Section 5.3), 
and sketch an ``elementary'' proof of the strong closing lemma for the standard contact $3$-sphere (Section 5.4). 
Finally we discuss the case of rational area-preserving maps (Section 5.5). 

\subsection{ECH spectral invariants} 

ECH (embedded contact homology) is a version of Floer homology for contact $3$-manifolds 
which was introduced by Hutchings as an analogue of the Gromov--Taubes invariant \cite{Taubes_Gromov} for closed $4$-dimensional symplectic manifolds. 
First we very briefly recall the construction of ECH. 
See the excellent introductory text \cite{Hutchings_lecturenote}. 

Let $(Y,\xi)$ be a closed contact $3$-manifold and let  $\lambda \in \Lambda_\nondeg(Y,\xi)$. 
For any $\gamma \in \mca{P}_\emb(Y, \lambda)$,
the number of real eigenvalues of the linearized return map of $\gamma$ is $0$ or $2$. 
In the first case $\gamma$ is called \textit{elliptic}, 
and in the second case $\gamma$ is called \textit{hyperbolic}. 
Let $\gen_\ech(Y,\lambda)$ 
denote the set 
which consists of finite subsets 
$\{(m_i, \gamma_i)\}_{i  \in I}$
of $\Z_{>0} \times \mca{P}_\emb(Y,\lambda)$
such that 
if $\gamma_i$ is hyperbolic then $m_i=1$. 

Although ECH can be defined with $\Z$-coefficients, 
for our purposes it is sufficient to consider $\Z/2$-coefficients. 
Let $C_\ech(Y,\lambda)$ denote the $\Z/2$-vector space generated by 
$\gen_\ech(Y,\lambda)$. 

For any element $x= \{(m_i,\gamma_i)\}_{i \in I}$ of $\gen_\ech(Y,\lambda)$, 
we define an one-dimensional current $C(x)$ on $Y$ 
by $C(x)(\alpha):= \sum_{i \in I} m_i \int_{\gamma_i} \alpha \quad  (\forall \alpha \in \Omega^1(Y))$. 
Also, we define $\mca{A}_\lambda(x):=C(x)(\alpha) \in \R$. 
For any $a \in \R$, 
let $C^a_\ech(Y,\lambda)$ 
denote the subspace of $C_\ech(Y,\lambda)$ 
which is generated over 
$\{ x \in \gen_\ech(Y,\lambda) \mid \mca{A}_\lambda(x)<a\}$. 

Let us define a boundary operator on $C_\ech(Y,\lambda)$.
Let $\mca{J}(Y,\lambda)$ denote the set of all almost complex structures on 
$Y \times \R$ satisfying the following conditions: 
\begin{itemize} 
\item $J$ is preserved by the natural action of $\R$ on $Y \times \R$. 
\item $J$ preserves $\xi$. Moreover $g_J(v,w):=d\lambda(v, Jw)$ is a positive-definite symmetric bilinear form on $\xi$. 
\item $J(\del_r)=R_\lambda$, where $r$ denotes the coordinate on $\R$. 
\end{itemize} 
For any $J \in \mca{J}(Y,\lambda)$, 
let $\tilde{\mca{M}}_{J, \inj}$ denote the set of quadruples 
$(\Sigma, P^+, P^-, u)$
satisfying the following conditions: 
\begin{itemize} 
\item $\Sigma$ is a connected and closed Riemann surface. Let $j_\Sigma$ denote the complex structure on $\Sigma$. 
\item $P^+$ and $P^-$ are disjoint finite subsets of $\Sigma$. 
\item $u: \Sigma \setminus (P^+ \cup P^-) \to Y \times \R$ is a $J$-holomorphic curve, i.e. there holds $J \circ du = du \circ j_\Sigma$. 
\item For any $p^+ \in P^+$, 
there exist 
a complex coordinate on $\Sigma$ such that $p^+$ corresponds to $0$
and $\gamma^+ \in \mca{P}(Y,\lambda)$, such that 
\[ 
\lim_{r \to \infty} u_\R(e^{-r - i \theta}) = \infty, \qquad
\lim_{r \to \infty} u_Y(e^{-r-i\theta}) = \gamma^+(\theta).
\] 
Here $u_\R$ denotes the $\R$-component of $u$, 
and $u_Y$ denotes the $Y$-component of $u$. 
\item For any $p^- \in P^-$, 
there exist 
a complex coordinate on $\Sigma$ such that $p^-$ corresponds to $0$
and 
$\gamma^- \in \mca{P}(Y,\lambda)$, such that 
\[ 
\lim_{r \to -\infty} u_\R(e^{r+i \theta}) = -\infty, \qquad 
\lim_{r \to -\infty} u_Y(e^{r+i\theta}) = \gamma^-(\theta). 
\] 
\item $u$ is somewhere injective. 
Namely, there exists $z \in \Sigma \setminus (P^+ \cup P^-)$ such that 
$u^{-1}(u(z))= \{z\}$
and $(du)_z: T_z\Sigma \to T_{u(z)}(Y \times \R)$ is injective. 
\end{itemize} 

Considering biholomorphic maps acting on domains of maps, 
one can define a natural equivalence relation on 
$\tilde{\mca{M}}_{J, \inj}$. 
Let $\mca{M}_{J, \inj}$ denote the quotient by this equivalence relation. 

\begin{remark} 
For any $\gamma \in \mca{P}_\emb(Y,\lambda)$, 
take its representative and denote it by $\tilde{\gamma}$. 
Then, a $J$-holomorphic map 
\[
\C \setminus \{0\} \to Y \times \R; \quad  e^{r+2\pi i\theta} \mapsto (r, \tilde{\gamma}(\theta))
\]
gives an element of $\mca{M}_{J, \inj}$. 
We denote this element by $\gamma \times \R$ and call it a trivial cylinder. 
\end{remark} 

For generic $J$, 
the moduli space $\mca{M}_{J,\inj}$ has a structure of a finite-dimensional manifold. 
Note that the dimension may depend on each connected component. 
Since $J$ is invariant by the $\R$-action, 
there exists a natural $\R$-action on $\mca{M}_{J,\inj}$. 
A point on $\mca{M}_{J,\inj}$ is fixed by this $\R$-action 
if and only if it is a disjoint union of trivial cylinders. 

For any $[u] \in \mca{M}_{J,\inj}$, 
the image of $u$ intersects $Y \times \{r\}$
transversally 
if $|r|$ is sufficiently large. 
Projecting the intersection to $Y$, 
we obtain a one-dimensional current on $Y$.
Let $C([u], r)$ denote this current. 
For any $x, y  \in \gen_\ech(Y,\lambda)$, 
let $\mca{M}_J(x,y)$ denote the set consisting of finite sets 
$\mca{C} = \{ (m_i , [u_i]) \}_{i \in I } \subset \Z_{>0} \times \mca{M}_{J, \inj}$ such that 
\[ 
\lim_{r \to \infty} \sum_{i \in I} m_i \cdot C([u_i], r) = x, \qquad
\lim_{r \to -\infty} \sum_{i \in I} m_i \cdot C([u_i],r) = y. 
\] 
Here $\lim$ denotes the limit as currents. 

For any $\mca{C} \in \mca{M}_J(x,y)$
one can define an integer $I(\mca{C})$ called ECH-index. 
For generic $J$ the following properties hold 
(see \cite{Hutchings_lecturenote} Proposition 3.7, Section 5.3, 5.4): 
\begin{itemize} 
\item $I(\mca{C}) \ge 0$ for any $\mca{C} \in \mca{M}_J(x,y)$. 
\item If $I(\mca{C})=0$ then $\mca{C}$ is a disjoint union of trivial cylinders. Namely $\mca{C} = \{ (m_i, \gamma_i \times \R) \}_{i\in I}$.  
\item If $I(\mca{C})=1$ then $\mca{C}=\mca{C}_0 \sqcup \{(1, C_1)\}$, 
where $I(\mca{C}_0)=0$ 
and $C_1$ is represented by an embedded $J$-holomorphic curve, 
and the dimension of $\mca{M}_{J,\inj}$ on a neighborhood of $C_1$ is $1$. 
Note that dimension of the quotient by the $\R$-action of this neighborhood is $0$. 
\item For any $x$ and $y$, $\{ \mca{C}  \in \mca{M}_J(x,y) \mid I(\mca{C})=1 \}/\R$ is a finite set. 
Let $\del_J(x,y) \in \Z/2$ denote the modulo $2$ cardinality of this set, 
and define the boundary operator $\del_J$ on $C_\ech(Y,\lambda)$ by 
\[ 
\del_J(x):= \sum_{y \in \gen_\ech(Y,\lambda)}  \del_J(x,y) \cdot y \qquad( x  \in \gen_\ech (Y,\lambda)). 
\]
Then $\del_J^2=0$. 
\end{itemize} 

Summing up, we have obtained a chain complex $C_\ech(Y,\lambda)$. 
Let $\ech(Y,\lambda)$ denote the homology of this complex. 
$\ech(Y,\lambda)$ depends only on $\xi$, 
not depending on the choice of $\lambda \in \Lambda_\nondeg(Y,\xi)$. 
Then we denote $\ech(Y,\lambda)$ by $\ech(Y,\xi)$. 
Actually, $\ech(Y,\xi)$ depends only on the homotopy type of $\xi$ as an oriented plane field. 
This is a consequence of the isomorphism between ECH and Seiberg-Witten cohomology, 
which is due to Taubes \cite{Taubes_ECH_SWF}. 
This isomorphism is an analogue of the fact that 
the Gromov--Taubes invariant of a closed symplectic four-manifold is equal to the corresponding Seiberg-Witten invariant, 
which is also due to Taubes \cite{Taubes_Gromov}. 

To define ECH spectral invariants we need the following lemma. 

\begin{lemma}\label{lem_action} 
Let $x, y \in \gen_\ech(Y, \lambda)$.
If 
$\mca{M}_J(x,y) \ne \emptyset$, then 
$\mca{A}_\lambda(x) \ge \mca{A}_\lambda(y)$. 
\end{lemma} 
\begin{proof} 
For any $\{  (m_i, [u_i]) \}_{i \in I} \in \mca{M}_J(x,y)$, 
there holds 
\[ 
\mca{A}_\lambda(x) - \mca{A}_\lambda(y) = \sum_{i \in I} m_i \cdot \int_{u_i} d \lambda 
\] 
by Stokes' theorem. 
Since $J$ preserves $\xi$ and $g_J$ is positive definite, 
for every $i$ there holds $\int_{u_i} d \lambda \ge 0$. 
Then we obtain RHS $\ge 0$. 
\end{proof} 

By this lemma, for any $a \in \R$, $C^a_\ech(Y,\lambda)$ is a subcomplex of $C_\ech(Y,\lambda)$. 
Let $\ech^a(Y,\lambda)$ denote the homology of $C^a_\ech(Y,\lambda)$, 
and let 
\[ 
i^a_\lambda: \ech^a(Y,\lambda) \to \ech(Y,\lambda) \cong \ech(Y,\xi)
\] 
be the natural map. These data are well-defined, not depending on $J$. 
Then, for any $\sigma \in \ech(Y,\xi) \setminus \{0\}$, 
one can define $s_\sigma: \Lambda_\nondeg(Y,\xi) \to \R_{\ge 0}$ by 
\[ 
s_\sigma(\lambda):= \inf\{ a  \mid \sigma \in \image i^a_\lambda\}. 
\] 
Then $s_\sigma$ satisfies spectrality, conformality, and monotonicity. 
The first two properties are straightforward. 
Monotonicity follows from the existence of ECH cobordism maps 
and the fact that the cobordism maps preserve action filtrations. 
See \cite{Hutchings_quantitative} and \cite{Hutchings_lecturenote}. 
As noted right after Definition \ref{defn_action_selector}, 
$s_\sigma$ uniquely extends to a spectral invariant 
$\Lambda(Y, \xi) \to \R_{\ge 0}$, 
which we also denote by $s_\sigma$. 

This construction of ECH spectral invariants is due to Hutchings \cite{Hutchings_quantitative}. 
In the same paper, 
he introduced a sequence of symplectic capacities 
called ECH capacities for four-dimensional Liouville domains (see Section 5.3 for the definition of Liouville domains) 
and demonstrated their applications to symplectic embedding problems. 

\subsection{Weyl law of ECH spectral invariants} 

Our first goal is to state the Weyl law of ECH spectral invariants. 
Let $(Y,\xi)$ be a closed and connected contact $3$-manifold, 
and let $\lambda \in \Lambda_\nondeg(Y,\xi)$. 
Let us define 
$[\, \cdot \, ]: \gen_\ech(Y, \lambda) \to H_1(Y:\Z)$ by 
$\big[ \{ (m_i, \gamma_i) \}_{i \in I} \big] := \sum_{i \in I}  m_i \cdot [\gamma_i]$. 
Here $[\gamma_i] \in H_1(Y:\Z)$ is the fundamental class of any representative of $\gamma_i$. 
For any $\Gamma \in H_1(Y:\Z)$, let us denote 
$\gen_\ech(Y,\lambda:\Gamma):= \{ x \in \gen_\ech(Y,\lambda) \mid [x] = \Gamma \}$, 
and let 
$C_\ech(Y, \lambda:\Gamma)$ be the subspace of $C_\ech(Y:\lambda)$ 
generated over 
$\gen_\ech(Y, \lambda:\Gamma)$.

For any $x, y \in \gen_\ech(Y, \lambda)$, if 
$\mca{M}_J(x,y) \ne \emptyset$ then $[x]=[y]$, 
thus 
$C_\ech(Y,\lambda:\Gamma)$
is a subcomplex of $C_\ech(Y,\lambda)$. 
Let $\ech(Y,\lambda:\Gamma)$ denote the homology of this complex. 
Let us define 
$d_\Gamma \in 2 \Z_{\ge 0}$ by 
\[ 
\{ \langle c_1(\xi) + 2 \mathrm{PD}(\Gamma), a \rangle \mid a \in H_2(Y:\Z) \}  = d_\Gamma \Z. 
\] 
Here $c_1$ denotes the Chern class, and 
$\mathrm{PD}$ denotes the Poincar\'{e} dual. 
For any $x,y \in \gen_\ech(Y,\lambda:\Gamma)$
one can define 
$I(x,y) \in \Z/d_\Gamma$, and if 
$\del_J(x,y)\ne 0$ then $I(x,y)=1$. 
Hence the chain complex $C_\ech(Y,\lambda:\Gamma)$
and the homology $\ech(Y, \xi:\Gamma)$ has relative $\Z/d_\Gamma$-grading. 

It is easy to verify $c_1(\xi) \in 2 H^2(Y:\Z)$, thus 
there exists $\Gamma \in H_1(Y:\Z)$ such that $d_\Gamma=0$. 
For such $\Gamma$, $\ech(Y, \xi:\Gamma)$ has relative $\Z$-grading.
Moreover, this relative grading is unbounded from above; 
indeed, there exists a sequence $(\sigma_k)_{k \ge 1}$ of nonzero homology classes 
such that $I(\sigma_{k+1}, \sigma_k)=2$ for all $k$. 
This is a consequence of the fact that 
the isomorphism defined by Taubes \cite{Taubes_ECH_SWF} 
between ECH and the Seiberg-Witten-Floer cohomology
preserves grading, and the corresponding result on the Seiberg-Witten-Floer cohomology, which is due to Kronheimer-Mrowka \cite{KM}. 
See \cite{CGH} Corollary 2.2 for details. 

Now we are ready to state the Weyl law. 

\begin{theorem}[Cristofaro-Gardiner--Hutchings--Ramos \cite{CGHR}]\label{thm_Weyl_law_ECH}
Let $(Y,\xi)$ be a closed and connected contact $3$-manifold, 
let $\Gamma \in H_1(Y:\Z)$ such that $d_\Gamma=0$,
and let $(\sigma_k)_{k \ge 1}$ 
be a sequence of nonzero homogeneous classes in $\ech(Y, \xi:\Gamma)$ 
such that $I(\sigma_{k+1}, \sigma_k)=2$ for all $k$. 
Then, for any $\lambda \in \Lambda(Y, \xi)$ 
\begin{equation}\label{eqn_Weyl_law} 
\lim_{k \to \infty}   \frac{s_{\sigma_k}(\lambda)^2}{2k} =  \vol(Y,\lambda):= \int_Y \lambda \wedge d\lambda. 
\end{equation} 
\end{theorem}

We can prove Theorem \ref{thm_local_sensitivity_dim3} as an immediate consequence. 
For any $\lambda \in \Lambda(Y, \xi)$ and $h \in C^\infty(Y, \R_{\ge 0}) \setminus \{0\}$, 
a simple computation shows that $\vol(Y, e^h \lambda) > \vol (Y, \lambda)$. 
Then $s_{\sigma_k}(Y, e^h\lambda) > s_{\sigma_k}(Y, \lambda)$ for sufficiently large all $k$. 
This implies that $\lambda$ satisfies local sensitivity. 

Theorem \ref{thm_Weyl_law_ECH} was conjectured in \cite{Hutchings_quantitative}.
The conjecture was proved in \cite{Hutchings_quantitative} 
for some special cases
(e.g. boundaries of star-shaped domains in $\C^2$). 
The proof for all contact $3$-manifolds was first given in \cite{CGHR}. 
Later, Sun \cite{Sun} gave an alternative proof with refined estimates.  

Let us briefly comment on the proofs of Theorem \ref{thm_Weyl_law_ECH}. 
For the boundaries of symplectic ellipsoids (in particular, balls), 
all values of ECH spectral invariants are explicitly computed in \cite{Hutchings_quantitative}, 
and the Weyl law follows directly from this computation. 
For general $(Y,\lambda)$, 
the lower bound 
$\liminf_{k \to \infty} s_{\sigma_k}(\lambda)^2/(2k) \ge \vol(Y,\lambda)$
follows from this result for balls and an interesting geometric argument called ball packing argument; see \cite{CGHR} for details. 
For some special cases (e.g. star-shaped domains in $\C^2$), 
one can also establish the upper bound 
$\limsup_{k \to \infty} s_{\sigma_k}(\lambda)^2/(2k) \le \vol(Y, \lambda)$
using ball packing argument; see \cite{Hutchings_quantitative}. 
In the general case, however, proving the upper bound requires 
a large perturbation (involving the contact form) of the Seiberg-Witten equations, 
which appeared in the Taubes's proof of the $3$-dimensional Weinstein conjecture; 
see \cite{CGHR} and \cite{Sun}. 

\subsection{Elementary spectral invariants} 

Hutchings introduced elementary capacities \cite{Hutchings_elementary} 
of $4$-dimensional Liouville domains 
and elementary spectral invariants \cite{Hutchings_elementary_closing} 
of contact $3$-manifolds. 
These invariants are elementary alternatives to ECH capacities/spectral invariants, 
and recover many applications of ECH, including strong closing lemmas. 
In this subsection, 
we discuss the definition and main properties 
of elementary spectral invariants of contact $3$-manifolds. 

Let  $(Y, \xi)$ be a closed contact $3$-manifold, 
and let $\lambda \in \Lambda_\nondeg(Y,\xi)$. 
For any $R \in \R_{>0}$, 
let $\mca{J}_R(Y, \lambda)$ denote the set of almost complex structures $J$ on $Y \times \R$ such that 
\begin{itemize} 
\item $J$ is compatible with the symplectic form $d(e^r\lambda)$ on $Y \times \R$, 
\item $J|_{Y \times \R_{\ge 0}}$ maps $\del_r$ to $R_\lambda$, and invariant by the $\R_{\ge 0}$-translation on $Y \times \R_{\ge 0}$, 
\item $J|_{Y \times \R_{\le -R}}$ maps $\del_r$ to $R_\lambda$, and invariant by the $\R_{\le 0}$-translation on $Y \times \R_{\le -R}$. 
\end{itemize} 

For any $J \in \mca{J}_R(Y,\lambda)$, one can define $\mca{M}_{J, \inj}$ similarly to Section 5.1. 
For any $u \in \mca{M}_{J, \inj}$, let $E_+(u):= \sum_\gamma T_\gamma$, 
where $\gamma$ runs over periodic Reeb orbits appearing on positive ends of $u$. 
Let $\mca{M}_J$ denote the set consisting of finite sets $u:=\{u_i\}_i$, 
where $u_i$ are distinct elements of $\mca{M}_{J, \inj}$. 
Let $E_+(u):= \sum_i E_+(u_i)$ and $\mathrm{Im}(u):= \bigcup_i \mathrm{Im}(u_i)$. 

For each $k \in \Z_{>0}$ and $R \in \R_{>0}$, let 
\begin{equation}\label{eqn_elementary_sRk} 
s^R_k(Y, \lambda):= \sup_{\substack{J \in \mca{J}_R(Y,\lambda) \\ x_1,\ldots,x_k \in Y \times [-R, 0]}}
\bigg( \inf_{\substack{ u \in \mca{M}_J \\ x_1, \ldots, x_k \in \mathrm{Im}(u)}} E_+(u) \bigg). 
\end{equation} 
This is nondecreasing on $R$, thus we can define the $k$-th elementary spectral invariant 
\[ 
s_k(Y, \lambda):= \lim_{R \to \infty} s^R_k(Y, \lambda). 
\]  
Actually this definition is slightly different from that in \cite{Hutchings_elementary_closing}, however is equivalent to it. 
We also set $s_0(Y,\lambda):=0$. 

Let us summarize main properties of the elementary spectral invariants. 
Let $\lambda \in \Lambda_\nondeg(Y, \xi)$. 
\begin{description} 
\item[(Conformality)] $s_k(Y,a\lambda) = s_k(Y,\lambda)$ for any $a \in \R_{>0}$ and $k \ge 0$. 
\item[(Increasing)] $0 < s_k(Y,\lambda) \le s_{k+1}(Y, \lambda)$ for any $k \ge 1$. 
\item[(Subadditivity)] $s_{k+l}(Y, \lambda) \le s_k(Y, \lambda) + s_l(Y, \lambda)$ for any $k, l \ge 0$. 
\item[(Spectrality)]  $s_k(Y, \lambda) \in \mca{A}_+(Y,  \lambda) \cup \{\infty\} $ for any $k \ge 0$. 
\item[(Monotonicity)]
$s_k(Y, \lambda) \le s_k(Y, e^h\lambda)$ for any $k \ge 0$ and $h \in C^\infty(Y, \R_{\ge 0})$ such that $e^h \lambda \in \Lambda_\nondeg(Y,\xi)$. 
\item[(Weyl Law)] $\lim_{k \to \infty} \frac{s_k(Y,\lambda)}{\sqrt{k}} = \sqrt{2 \mathrm{vol}(Y, \lambda)}$. 
\item[(Slow growth)] $\lim_{k \to \infty} \frac{s_k(Y,\lambda)}{k}=0$. 
\end{description} 

The first four  properties are straightforward from the definition. 
Monotonicity is proved in \cite{Hutchings_elementary_closing} Theorem 1.14. 
It follows from the so-called neck-stretching argument, 
however the issue is that one cannot directly apply the standard SFT-compactness theorem \cite{BEHWZ} 
due to the lack of a priori genus bound. 
This issue can be resolved with 
special properties in low dimensions (the relative adjunction formula and the asymptotic writhe bound), 
both proved in \cite{Hutchings_revisited}. 
The Weyl law follows from the 
comparison with the ECH spectral invariants (\cite{Hutchings_elementary_closing} Theorem 6.1), 
the Weyl law for the ECH spectral invariants, 
and the lower bound estimate proved by the ball packing argument (see \cite{Hutchings_elementary_closing} Theorem 1.14). 
Note that the Weyl law in particular implies $s_k(Y, \lambda)< \infty$ for all $k$, which is nontrivial. 
The slow growth property follows immediately from the Weyl law. 

So far the invariant $s_k$ is defined only for nondegenerate contact forms. 
It uniquely extends to all (possibly degenerate) contact forms so that the monotonicity holds. 
Clearly each $s_k$ is a spectral invariant, 
and the Weyl law implies that all contact forms on contact $3$-manifolds satisfy the strong closing property. 

Although the slow growth property seems much weaker than the Weyl law, 
it is in fact sufficient to establish local sensitivity (thus the strong closing lemma), with the help of the following lemma. 

\begin{lemma}\label{lem_ball_inserting}  
For any $h \in C^\infty(Y, \R_{\ge 0}) \setminus \{0\}$, there exists $\ep>0$ such that for any $k \in \Z_{\ge 0}$ 
\begin{equation}\label{eqn_ball_inserting}
s_{k+1} (Y, e^h\lambda) - s_k(Y, \lambda) 
\ge \ep. 
\end{equation} 
\end{lemma} 
\begin{proof}[Proof (sketch)] 
Take $\ep>0$ so that 
$B^4(\ep):=E_{(\ep, \ep)}$ (in the notation of Example \ref{ex_ellipsoids}) 
symplectically embeds into
$Y_h:= \{ (y,r) \in Y \times \R \mid 0 < r< h(y)\}$ 
with the symplectic form $d(e^r\lambda)$. 
Then, (\ref{eqn_ball_inserting}) 
follows from 
the neck stretching argument 
and the monotonicity lemma for $J$-holomorphic curves. 
\end{proof}

Lemma \ref{lem_ball_inserting} implies the following corollary. 

\begin{corollary} \label{cor_spectral_gap_criterion} 
If $\inf_k s_{k+1}(Y,\lambda) - s_k(Y,\lambda)=0$ then 
$\lambda$ satisfies the local sensitivity. 
\end{corollary} 

The slow growth property implies 
$\inf_k s_{k+1}(Y,\lambda) - s_k(Y, \lambda)=0$, 
thus $\lambda$ satisfies local sensitivity. 
It is not difficult to prove results similar to Corollary \ref{cor_spectral_gap_criterion} in more general settings, 
see e.g. \cite{Irie_SCP}. 
Also, a result similar to Corollary \ref{cor_spectral_gap_criterion} plays a key role in the work by Chaidez-Tanny \cite{Chaidez_Tanny} as we discuss it in Section 7.2. 

\begin{remark}\label{rem_quantitative_closing_lemma} 
Using the above arguments it is not difficult to show the following: 
\begin{quote} 
Let $k \ge 1$ and $r_k:= s_k(Y, \lambda)/k$. 
If $B^4(r_k)$ symplectically embeds into 
$(Y_h, d(e^r\lambda))$, 
then there exist $t \in [0,1]$ and $\gamma \in \mca{P}(e^{th}\lambda)$ 
such that $\image(\gamma) \cap \supp(h) \ne \emptyset$ and $T_\gamma \le s_k(Y, \lambda)$. 
\end{quote} 
This observation leads to quantitative closing lemmas, i.e. closing lemmas with upper bounds of periods of orbits. 
This idea is extensively investigated by Hutchings \cite{Hutchings_elementary_closing}. 
Quantitative closing lemmas for area-preserving maps are studied in \cite{Edtmair_Hutchings} and \cite{Edtmair_elementary}. 
Furthermore, 
this observation 
plays an important role in proving the 
$C^\infty$-closing lemma for $3$-dimensional Reeb flows in real-analytic setting; see \cite{Irie_KJM}. 
\end{remark} 

So far we have discussed elementary spectral invariants of contact $3$-manifolds. 
One can also define elementary capacities $(c_k)_{k \ge 0}$ of $4$-dimensional Liouville domains. 
A Liouville domain is a pair $(X,\lambda)$ such that $X$ is a $2n$-dimensional compact manifold with boundary, 
$\lambda \in \Omega^1(X)$, $d\lambda$ is symplectic and the vector field $Z \in \mca{X}(X)$ defined by $L_Z(d\lambda)=\lambda$
points strictly outwards at every point on $\del X$. 
The definition of the elementary capacities is similar to that of spectral invariants. 
See \cite{Hutchings_elementary} for the definition. 
An important difference is that it is not known whether
$c_1(X,\lambda)<\infty$ for every $4$-dimensional Liouville domain $(X,\lambda)$. 
In particular, the Weyl law is established only for some limited class of $4$-dimensional Liouville domains; see \cite{Hutchings_elementary}. 

Let us consider whether the argument in this section extends to 
high dimensions, i.e. to contact manifolds of dimension $\ge 5$ and Liouville domains of dimension $\ge 6$. 
The definition given by the formula (\ref{eqn_elementary_sRk}) makes sense, 
and the first four properties (conformality, increasing, subadditivity, and spectrality) hold. 
As already noted, the proof of monotonicity relies on special properties in low dimensions. 
This difficulty, however, can be circumvented by slightly modifying the definition of elementary spectral invariants. 
In contrast, it seems that the Weyl law is a genuinely special phenomenon in low dimensions, 
and slow growth property fails in high dimensions. 
The definition and properties of elementary spectral invariants in high dimensions are discussed in Section 7.1. 

\subsection{Slow growth property for the standard contact $3$-sphere} 

As far as the author knows, 
to prove the slow growth property of elementary spectral invariants 
for general contact $3$-manfolds, 
one requires the comparison with the ECH spectral invariants 
and the upper bound estimate of the ECH spectral invariants using 
perturbed Seiberg-Witten equations. 
However, for some nice contact $3$-manifolds, 
one can prove the slow growth property using only elementary knowledge of pseudo-holomorphic curve theory. 
The goal of this subsection is to sketch such a proof for contact forms on the standard contact $3$-sphere
(equivalently, for star-shaped hypersurfaces in $\R^4$). 
It is not difficult to extend the following proof to more general cases, for example restricted contact type hypersurfaces in $\R^4$. 

Let us state our goal as the following proposition. 

\begin{proposition}\label{prop_slowgrowth_S3} 
Let $\xi_{\mathrm{std}}$ denote the standard contact structure on $S^3$. 
For any $\lambda \in \Lambda(S^3, \xi_{\mathrm{std}})$, there holds 
$\lim_{k \to \infty} s_k(S^3, \lambda)/k = 0$. 
\end{proposition} 

We first define the elementary capacity of closed $4$-dimensional symplectic manifolds. 
Let $(M, \omega)$ be a closed $4$-dimensional symplectic manifold, 
and let $\mca{J}(M,\omega)$ be the set of almost complex structures on $M$ compatible with $\omega$. 
For any $J \in \mca{J}(M, \omega)$, let $\mca{M}_J(M, \omega)$ denote the set consisting of pairs
$(u,\Sigma)$ where 
$\Sigma$ is a (possibly disconnected) compact Riemann surface (with the complex structure $j_\Sigma$), 
$u: \Sigma \to M$ satisfies $J \circ du = du \circ j_\Sigma$, 
and the restriction of $u$ to each connected component of $\Sigma$ is nonconstant. 
Let $E(u):= \int_\Sigma u^*\omega$. 
For each $k \in \Z_{>0}$, we define 
\[ 
c_k(M, \omega):=
 \sup_{\substack{J \in \mca{J}(M, \omega) \\ x_1, \ldots, x_k \in M}} 
\bigg( \inf_{\substack{u \in \mca{M}_J(M,\omega) \\x_1, \ldots, x_k \in  \image(u)}}  E(u)  \bigg). 
\] 

The proof of the next lemma is parallel to that of the monotonicity property. 

\begin{lemma}\label{lem_monotonicity_closed} 
Let $(Y, \xi)$ be a closed contact $3$-manifold, and $\lambda \in \Lambda(Y, \xi)$. 
If $(Y \times \R_{\le 0}, d(e^r\lambda))$ symplectically embeds into $(M, \omega)$, 
then $s_k(Y, \lambda) \le c_k(M, \omega)$ for all $k$. 
\end{lemma} 

By Lemma \ref{lem_monotonicity_closed}, 
Proposition \ref{prop_slowgrowth_S3} is reduced to show $\lim_{k \to \infty} c_k(\C P^2)/k = 0$. 
This follows from the following inequality: 
\begin{equation}\label{CP2_upperbound} 
c_{d(d+3)/2}(\C P^2) \le d \qquad(\forall d \in \Z_{>0}). 
\end{equation} 

Actually, Hutchings \cite{Hutchings_elementary} explicitly computed $c_k(\C P^2)$ for all $k$: 
\[ 
c_k(\C P^2) = \min \{ d \mid d \in \Z_{>0}, \,  d (d+3)  \ge 2k \}. 
\] 
In particular, the inequality in (\ref{CP2_upperbound}) can be replaced with the equality. 
The proof of (\ref{CP2_upperbound}) in \cite{Hutchings_elementary} uses Taubes's ``Seiberg-Witten $=$ Gromov'' theorem \cite{Taubes_Gromov}. 
Edtmair \cite{Edtmair_elementary} gave a more elementary proof, though it still uses the Gromov--Taubes invariant. 
Here we sketch an even more elementary proof of (\ref{CP2_upperbound}). 

By the Gromov compactness theorem, it is sufficient to show the following. 
For each $k \in\Z_{>0}$, let $S_k(\C P^2)$ denote the set consisting of 
$x=(x_1, \ldots, x_k)$, where $x_1, \ldots, x_k$ are distinct points on $\C P^2$. 

\begin{proposition}\label{prop_degree_d_curves} 
Let $d \in \Z_{>0}$ and $x \in S_{d(d+3)/2}(\C P^2)$. 
For generic $J \in \mca{J}(\C P^2)$, 
there exists $(u, \Sigma) \in \mca{M}_J(\C P^2)$ 
such that $\Sigma$ is a connected Riemann surface of genus $(d-1)(d-2)/2$, 
$u_*[\Sigma] = d [\C P^1]$ and $x_1, \ldots, x_{d(d+3)/2} \in u(\Sigma)$. 
\end{proposition} 
\begin{proof}[Proof (sketch)] 
Let $k:= d(d+3)/2$. 
For any $J \in \mca{J}(\C P^2)$ and $x \in S_k(\C P^2)$, 
consider the moduli space $\mca{M}_J (x)$ which consists of 
equivalence classes of tuples $(u, \Sigma, p_1, \ldots, p_k)$ such that: 
\begin{itemize} 
\item $(u, \Sigma) \in \mca{M}_J(\C P^2)$ and $u_*[\Sigma]=d[\C P^1]$. 
\item $\Sigma$ is a connected Riemann surface of genus $(d-1)(d-2)/2$. 
\item $u$ is somewhere injective.
\item $p_1, \ldots, p_k \in \Sigma$ such that $u(p_j)=x_j$ for all $1 \le j \le k$. 
\end{itemize} 

Let $J^0$ denote the standard complex structure on $\C P^2$. 
Then, for generic $x^0 \in S_k(\C P^2)$, 
$\mca{M}_{J^0} (x^0)$ consists of a unique element which is cut out transversally. 
This is an elementary fact in complex geometry: see \cite{Edtmair_elementary} Theorem B.3 for a detailed exposition. 

For an arbitrary $x \in S_k(\C P^2)$, 
take a path $(x^t)_{t \in [0,1]}$ on $S_k(\C P^2)$ 
such that $x^0$ is given as above and $x^1=x$. 
It is sufficient to show that, for a generic path $(J^t)_{t \in [0,1]}$ on $\mca{J}(\C P^2)$ 
such that $J^0$ is given as above, $\mca{M}_{J^1}(x^1) \ne \emptyset$. 

For a generic path $(J^t)_{t \in [0,1]}$, 
the moduli space $\tilde{\mca{M}}:= \bigcup_{t \in [0,1]} \mca{M}_{J^t} (x^t)$
is a $1$-dimensional manifold with boundary $\mca{M}_{J^0}(x^0) \sqcup \mca{M}_{J^1}(x^1)$. 
If $\tilde{\mca{M}}$ is not compact, then for some $t \in [0,1]$ 
there exists a broken or multiply-covered $J^t$-holomorphic curve which passes through $x^t_1, \ldots, x^t_k$. 
However, virtual dimensions of these broken or multiply-covered curves (with $x^t$ and $J^t$ fixed) are at most $-2$, 
thus these curves cannot appear for a generic path $(J^t)_{t \in [0,1]}$. 
For such a path $\tilde{\mca{M}}$ is compact, thus the cardinality of 
$\mca{M}_{J^1}(x^1)$ is odd, in particular nonzero. 
\end{proof} 

\begin{remark} 
A result essentially equivalent to 
Proposition \ref{prop_degree_d_curves} 
was already stated by Gromov \cite{Gromov_pseudoholomorphic}; 
see Section 0.2.$\mathrm{B}$ and a sketched proof in Section 2.4.$\mathrm{B}''_1$. 
\end{remark} 

\subsection{Spectral invariants for area-preserving maps} 

Periodic Floer homology (PFH) is an analogue of ECH for (mapping tori of) area-preserving maps of closed surfaces. 
Spectral invariants defined from PFH satisfy 
analogues of the ECH Weyl law; 
see \cite{CGPZ} Theorem 1.5 and \cite{Edtmair_Hutchings} Theorem 8.1, 
as well as \cite{CGPZ} Section 1.1.4 for the comparison of the two results. 
An analogue of the elementary spectral invariants for area-preserving maps 
is investigated in \cite{Edtmair_elementary}. 
Note that applications to closing lemmas (\cite{Edtmair_elementary} Section 1.6) 
are limited to the case of Hamiltonian diffeomorphisms.

\section{Generic equidistribution of periodic orbits} 

As we have explained, 
local sensitivity of spectral invariants is sufficient to prove the strong closing lemma (and the $C^\infty$-generic density of perioidic orbits). 
However, using the Weyl law we can prove the following stronger result. 

\begin{theorem}[\cite{Irie_JSG}]\label{thm_equidistribution} 
Let $(Y,\xi)$ be a closed contact $3$-manifold. 
For $C^\infty$-generic $\lambda \in \Lambda(Y,\xi)$, 
there exists a sequence $(C_N)_{N \ge 1}$ of $1$-dimensoional currents on $Y$ such that: 
\begin{itemize} 
\item[(i):] Each $C_N$ is a finite sum $C_N= \sum_j a_j\gamma_j$, where $a_j \in \R_{>0}$ and $\gamma_j \in \mca{P}_\emb(Y,\lambda)$. 
\item[(ii):]  $(C_N)_N$ weakly converges to $d\lambda$. Namely, for any $\alpha \in \Omega^1(Y)$, $\lim_{N \to \infty} C_N(\alpha) = \int_Y \alpha \wedge d\lambda$.
\end{itemize} 
\end{theorem} 

Several remarks on Theorem \ref{thm_equidistribution} are in order. 

\begin{itemize} 
\item The condition (ii) implies that $\bigcup_N \supp(C_N)$ is dense in $Y$. 
In particular, Theorem \ref{thm_equidistribution} implies the generic density of periodic oribits.
\item Theorem \ref{thm_equidistribution} is inspired by the result by Marques-Neves-Song \cite{MNS} for minimal hypersurfaces. 
Also, the proof in \cite{Irie_JSG} closely follows the argument in \cite{MNS}. 
\item Each $C_N$ may consist of (finitely) many periodic orbits. 
Indeed, in the situation where an elliptic periodic orbit is surrounded by KAM tori, 
$C_N$ with large $N$ consists of many orbits, since each periodic orbit cannot cross KAM tori. 
\item Colin-Dehornoy-Hryniewicz-Rechtman \cite{CDHR} proved that 
if a sequence $(C_N)_N$ satisfying the conditions (i) and (ii) exists 
then the flow of $R_\lambda$ admits a Birkhoff section (see \cite{CDHR} for the definition). 
Hence Theorem \ref{thm_equidistribution} implies that a $C^\infty$-generic contact form admits a Birkhoff section. 
\item An analogue of Theorem \ref{thm_equidistribution} for area-preserving surface maps is proved by Prasad \cite{Prasad_equidistribution}. 
It would be interesting to see whether this result has applications in surface dynamics. 
\end{itemize} 

Let us briefly discuss the key idea in the proof of Theorem \ref{thm_equidistribution}. 
For each $N \in \Z_{>0}$
we consider an appropriate $C^\infty$-family of contact forms $(\lambda_\tau)_{\tau \in [0,1]^N}$ in $\Lambda(Y, \xi)$. 
Let $(\sigma_k)_k$ be a sequence of ECH homology classes as in the Weyl law, 
and define functions 
$(e_k)_k$ on $[0,1]^N$ by $e_k(\tau):= s_{\sigma_k}(\lambda_\tau)/\sqrt{k} - \sqrt{2\vol(\lambda_\tau)}$. 
Using the Weyl law, we can show $\lim_{k \to \infty} \|e_k\|_{C^0}=0$.
If each $e_k$ were $C^1$ 
(equivalently, if $s_{\sigma_k}(\lambda_\tau)$ were $C^1$ in $\tau$), 
then for large $k$ 
one can find $\tau_k \in [0,1]^N$ such that 
$|\nabla e_k(\tau_k)|$ is small. 
Taking $C_N$ to be the sum of Reeb orbits representing $s_{\sigma_k}(\lambda_{\tau_k})$, 
the sequence $(C_N)_N$ satisfies the required conditions. 
However, the assumption that $s_{\sigma_k}(\lambda_\tau)$ is $C^1$ is unrealistic, 
since each $s_{\sigma_k}$ is defined as a min-max value, 
and it is easy to see that an analogous statement for finite-dimensional toy models is false. 
The key idea, due to \cite{MNS}, is to resolve this issue using \cite{MNS} Lemma~3, 
which is an elementary result on $C^0$-small Lipschitz functions on $[0,1]^N$. 

\section{Strong closing lemmas in high-dimensional Hamiltonian dynamics} 

It is natural to ask whether strong closing lemmas hold in high dimensions, 
i.e. Reeb flows on contact manifolds of dimension $\ge 5$ or 
rational symplectic/Hamiltonian diffeomorphisms of symplectic manifolds of dimension $\ge 4$. 
This question is wide open. 
We explain why an attempt to prove strong closing lemmas using elementary spectral invariants fails in high dimensions (Section 7.1), 
and review recent results 
establishing strong closing lemmas for very special high-dimensional Hamiltonian systems (Section 7.2). 

\subsection{Elementary spectral invariants in high dimensions} 

In this subsection we define elementary spectral invariants 
for contact manifolds of all dimensions, 
and show that the slow growth property fails in high dimensions. 
Details of the arguments in this subsection will appear on \cite{Irie_elementary}. 

Let $n \in \Z_{>0}$ 
and let $(Y,\xi)$ be a closed contact manifold of dimension $2n-1$. 
Let $\lambda \in \Lambda_{\nondeg} (Y, \xi)$. 
For any $R \in \R_{>0}$, 
$J \in \mca{J}_R(Y,\lambda)$, and $g \in \Z_{\ge 0}$, 
let $\mca{M}_{J, \inj}^{\le g}$ 
denote the set consisting of 
$(\Sigma, P^+, P^-, u) \in \mca{M}_{J, \inj}$ 
such that the genus of $\Sigma$ is at most $g$. 
Note that $\mca{J}_R(Y,\lambda)$ and $\mca{M}_{J, \inj}$ make sense in any dimension. 
Let $\mca{M}_J^{\le g}$ 
denote the set consisting of finite sets $u=\{u_i\}_i$ 
such that $u_i$ are distinct elements of $\mca{M}_{J, \inj}^{\le g}$. 
For each $k \in \Z_{>0}$, we define 
\[ 
s^{R, \le g}_k(Y, \lambda)
:= \sup_{\substack{ J \in \mca{J}_R(Y,\lambda)\\ x_1, \ldots, x_k \in Y \times [-R, 0] }} \bigg( 
\inf_{\substack{u \in \mca{M}^{\le g}_J \\ x_1, \ldots, x_k  \in \image(u) }} E_+(u) \bigg). 
\] 
Note that $E_+(u)$ and $\image(u)$ make sense in any dimension. 

$s^{R, \le g}_k$ is nonincreasing on $g$ and nondecreasing on $R$, thus we can define 
the $k$-th elementary spectral invariant
\[
\bar{s}_k(Y,\lambda):= \lim_{R \to \infty} \bigg( \lim_{g \to \infty} s^{R, \le g}_k(Y, \lambda) \bigg). 
\] 
We also set $\bar{s}_0(Y, \lambda):= 0$. 

\begin{remark} 
It turns out that $\bar{s}_k(Y,\lambda)=s_k(Y,\lambda)$ when $\dim Y=3$. 
The inequality $s_k \le \bar{s}_k$ is straightforward from the definition. 
The opposite inequality $s_k \ge \bar{s}_k$ is nontrivial, 
and one can prove this by an argument similar to the proof of the monotonicity of $s_k$, 
using the relative adjunction formula and the writhe bound. 
\end{remark} 

The invariants $\bar{s}_k$ 
satisfy the properties
conformality, increasing, subadditivity, spectrality, and monotonicity. 
Properties other than the monotonicity are straightforward. 
To prove the monotonicity, 
it is sufficient to prove the monotonicity of 
$s^{R, \le g}_k$ for all $k$, $R$, and $g$. 
The proof is a standard application of the neck-stretching argument, 
with the following two key observations: 
(1) one can apply the SFT-compactness theorem \cite{BEHWZ} since we have a priori genus bound, 
(2) if a connected surface $S$ is embedded into another connected surface $S'$, then the genus of $S$ is not larger than the genus of $S'$. 

One can define the invariants $\bar{s}_k$ for all (possibly degenerate) contact forms so that the monotonicity holds. 
One can also define elementary capacities $(\bar{c}_k)_{k \ge 0}$ for all Liouville domains in a similar manner. 
Now we show that the slow growth property
---which is the key to prove the strong closing lemma as explained in Section 5.3---
fails in high dimensions. 

\begin{theorem}[\cite{Irie_elementary}]\label{thm_growth_in_high_dimension} 
Let $n \in \Z_{\ge 3}$. 
Then $\lim_{k \to \infty} \bar{s}_k(Y,\lambda)/k>0$ for any $2n-1$-dimensional closed contact manifold $(Y,\xi)$ and $\lambda \in \Lambda(Y,\xi)$. 
Moreover, $\lim_{k \to \infty} \bar{c}_k(X,\lambda)/k>0$ for any $2n$-dimensional Liouville domain $(X,\lambda)$. 
\end{theorem} 
\begin{proof}[Proof (sketch)]
By the monotonicity, 
combined with Darboux's theorem, 
it is sufficient to show $\lim_{k \to \infty} \bar{c}_k(E_a, \lambda_n)/k>0$ for some $a \in A_n$ (see Example \ref{ex_ellipsoids}). 
If $a_i/a_j \not\in \Q$ for all $i<j$, this follows from 
computations of virtual dimensions of moduli spaces of holomorphic curves asymptotic to periodic Reeb orbits on positive ends. 
Details will appear on \cite{Irie_elementary}. 
\end{proof}

Theorem \ref{thm_growth_in_high_dimension}  implies that the argument in Section 5.3 does not extend to contact manifolds of dimension $\ge 5$. 
Nevertheless, we could still try to prove some weak versions of strong closing property, which we discuss here. 
Let $(Y, \xi)$ be a closed contact manifold.
For any $\alpha \in H_*(Y) \setminus \{0\}$, 
we say $\lambda \in \Lambda(Y, \xi)$ satisfies 
\textit{$\alpha$-strong closing property} 
if the following holds: 
\begin{quote} 
Let $C$ be a (singular) cycle on $Y$ representing $\alpha$. 
For any  
$h \in C^\infty(Y, \R_{\ge 0})$ such that $\min h|_{\image(C)}>0$, 
there exist $t \in [0,1]$ and $\gamma \in \mca{P}(Y, e^{th}\lambda)$ 
such that $\image(\gamma) \cap \supp (h) \ne \emptyset$. 
\end{quote} 

Let $[\mathrm{pt}] \in H_0(Y)$ denote the point class. 
Then $[\mathrm{pt}]$-strong closing property is equivalent to the strong closing property. 
In the following three cases, 
every $\lambda \in \Lambda(Y,\xi)$ satisfies $\alpha$-strong closing property by the obvious reasons:
\begin{itemize} 
\item[(i):] $\dim Y \le 3$ and $\alpha$ is arbitrary. 
\item[(ii):] $\alpha = [Y]$ and the Weinstein conjecture for $(Y,\xi)$ is verified. 
\item[(iii):] $|\alpha|=\dim Y - 1$, there exists $\beta \in H_1(Y)$ such that $\alpha \cdot  \beta \ne 0$  
and it is verified that  for any $\lambda \in \Lambda(Y, \xi)$ there exist $\gamma_1,\ldots, \gamma_j  \in \mca{P}(Y,\lambda)$
such that $[\gamma_1] + \cdots + [\gamma_j]=\beta$. 
\end{itemize} 

Apart from these three cases, establishing the $\alpha$-strong closing property seems to be highly nontrivial. 
As we cannot formulate any reasonable conjecture, we just propose the following problem. 

\begin{prob} 
Find a closed contact manifold $(Y, \xi)$ and $\alpha \in H_*(Y) \setminus \{0\}$ 
such that $\dim Y \ge \max\{5, |\alpha|+2\}$ and 
every $\lambda \in \Lambda(Y, \xi)$ satisfies $\alpha$-strong closing property. 
\end{prob} 

In another direction, we could try to prove strong closing property for some specific contact forms. 
In the next subsection we discuss recent progress in this direction. 

\subsection{Near-periodic Hamiltonian systems} 

We explain recent results 
which claim that ``nearly periodic'' Hamiltonian systems satisfy strong closing property. 
Let us first discuss the case of Reeb flows. 

\begin{definition}[\cite{Chaidez_Tanny} Definition 5.18] 
Let $Y$ be a closed manifold and $\lambda$ be a contact form on $Y$. 
$\lambda$ is called $T$-periodic (where $T \in \R_{>0}$) if 
$\ph^T_\lambda = \id_Y$. 
$\lambda$ is called \textit{Hofer near-periodic} 
if there exist sequences
$(h_i)_{i \ge 1}$ in $C^\infty(Y)$ and $(T_i)_{i \ge 1}$ in $\R_{>0}$ 
such that, 
for each $i$ the contact form $\lambda_i:= e^{h_i} \lambda$ is $T_i$-periodic 
and $\lim_{i \to \infty} \| h_i \|_{C^0} \cdot T_i = 0$. 
\end{definition} 

\begin{example} 
The canonical contact form on the boundary of any symplectic ellipsoid (Example \ref{ex_ellipsoids}) is Hofer near-periodic, 
as proved by Chaidez--Datta--Prasad--Tanny \cite{CDPT}. The proof uses Dirichlet approximation theorem. 
\end{example} 

\begin{example} 
Suppose $\lambda$ is a contact form on a closed manifold $Y$ such that 
$R_\lambda$ generates a free $S^1$-action, 
and suppose that the quotient $B:=Y/S^1$ admits a Hamiltonian $S^1$-action generated by a Morse function on $B$. 
The basic example is that $Y$ is the unit sphere cotangent bundle of $S^n$ with the round metric and $n \in \Z_{\ge 2}$. 
In this setting, Albers-Geiges-Zehmisch \cite{AGZ} produced a family of contact forms on $Y$, 
generalizing Katok's construction of Finsler metrics on $S^n$ with only finitely many closed geodesics \cite{Katok_finsler}. 
These contact forms are Hofer near-periodic, as proved in \cite{Chaidez_Tanny} Lemma 5.22. 
\end{example} 

Now we have the following theorem. 

\begin{theorem}[Chaidez--Tanny \cite{Chaidez_Tanny}]\label{thm_CT} 
Assume the existence of Gromov-Witten theory and SFT neck stretching theorem for symplectic orbifolds 
(see \cite{Chaidez_Tanny} Assumption 1.3). 
Then, any Hofer near-periodic contact form satisfies the strong closing property.
\end{theorem} 

This result follows from \cite{Chaidez_Tanny} Theorems 5 and 7. 
These results are formulated using 
ESFT (elementary SFT) spectral gaps, which are introduced in \cite{Chaidez_Tanny}. 
Each ESFT spectral gap is an invariant of contact forms taking values in $[0,\infty]$, 
defined similarly to elementary spectral invariants. 
One might think them as ``relative'' spectral invariants in the sense that they play a role similar to $s_{k+1}-s_k$ in Corollary \ref{cor_spectral_gap_criterion}. 
\cite{Chaidez_Tanny} Theorem 5 is an analogue of Corollary \ref{cor_spectral_gap_criterion}, 
and it claims that if the infimum of ESFT spectral gaps of a contact form is zero, then the contact form satisfies the strong closing property. 
\cite{Chaidez_Tanny} Theorem 7 claims that 
for any Hofer near-periodic contact form, the infimum of ESFT spectral gaps is zero. 
To prove this we need upper bounds of ESFT spectral gaps, 
which follow from nonvanishing of Gromov-Witten invariants of certain symplectic orbifolds. 
\cite{Chaidez_Tanny} Assumption 1.3 is needed in this process.  

\begin{remark} 
\cite{Chaidez_Tanny} works on the setting of conformal Hamiltonian manifolds satisfying the rationality condition. 
This is a general framework including contact manifolds and mapping tori of rational symplectic diffeomorphisms.
\cite{Chaidez_Tanny} Theorems 5 and 7 are formulated  in this setting. 
\end{remark} 

Theorem \ref{thm_CT} in particular implies that 
the canonical contact form on the boundary of any symplectic ellipsoid satisfies the strong closing property. 
This was conjectured in \cite{Irie_SCP}, 
and earlier proved in \cite{CDPT} without any assumption. 
This result also follows from \cite{Cineli_Seyfaddini} which we discuss below
(see \cite{Cineli_Seyfaddini} Section 6). 
Moreover, Xue \cite{Xue}
used KAM normal forms 
to prove strong closing lemmas (which are not exactly equivalent to the one formulated in the present article) 
for integrable Hamiltonian systems 
including the above examples on ellipsoids. 

Let us discuss the result by Cineli-Seyfaddini \cite{Cineli_Seyfaddini}
for Hamiltonian diffeomorphisms. 
For any closed symplectic manifold $(M,\omega)$, 
$\Ham(M,\omega)$ is equipped with a bi-invariant metric 
called spectral metric (or $\gamma$-metric) 
defined via a spectral invariant associated to the Hamiltonian Floer homology; see \cite{Cineli_Seyfaddini} Section 4. 
$\psi \in \Ham(M, \omega)$ is called \textit{$\gamma$-rigid} 
if there exists a sequence $(k_i)_{i \ge 1}$ of positive integers such that 
$\lim_{i \to \infty} k_i=\infty$ 
and $\lim_{i \to \infty} \psi^{k_i} \to \id_M$ with respect to the $\gamma$-metric. 

\begin{example} 
An element of a Hamiltonian torus action on $M$ is called a \textit{rotation}.  
Any rotation $\psi$ is $\gamma$-rigid, because there exists a sequence $(k_i)_i$ such that 
$\lim_{i \to \infty} \psi^{k_i}=\id_M$ in the $C^\infty$-topology, 
and $\gamma$ is continuous with respect to the $C^\infty$-topology. 
\end{example} 

\begin{example} 
$\psi \in \Ham(M,\omega)$ is called a \textit{pseudo-rotation} if it has finitely many periodic points. 
All pseudo-rotations of $\C P^n$ (with the standard Fubini-Study form), 
as well as Anosov-Katok pseudo-rotations are $\gamma$-rigid; 
see \cite{Cineli_Seyfaddini} Section 1.2 and the references therein. 
\end{example} 

The following is the main theorem of \cite{Cineli_Seyfaddini}. 

\begin{theorem}[\cite{Cineli_Seyfaddini}, Theorem 1]\label{thm_CS}
Suppose that $\psi \in \Ham(M,\omega)$ is $\gamma$-rigid. 
Then, for any $G \in \mca{H}(M) \setminus \{0\}$ such that $G \ge 0$ or $G \le 0$
and $\supp(G) \cap \mca{P}(\psi)=\emptyset$, 
there holds $\mca{P}(\ph^1_G \circ \psi) \cap \supp (G) \ne \emptyset$. 
\end{theorem} 

\begin{corollary} 
Any $\gamma$-rigid map in $\Ham(M,\omega)$ satisfies the strong closing property. 
\end{corollary} 

Theorem \ref{thm_CS} is surprising by the following two reasons: 
(1) the result is stronger than the strong closing property, which considers $1$-parameter families of perturubations, 
(2) the proof uses only spectral invariants defined from Hamiltonian Floer homology. 
While Theorem \ref{thm_CS} is very interesting, several results suggest that $\gamma$-rigid Hamiltonian maps are rather scarce
(see \cite{Cineli_Seyfaddini} Section 1.3 and the references therein). 
The author expects that Hofer near-periodic contact forms are also scarce,
but is not aware of any specific result addressing this. 

\section{Minimal hypersurfaces} 

In the recent advances in the theory of minimal hypersurfaces (see surveys \cite{Marques_JJM} and \cite{Zhou}), 
there is a story which is very similar to the one we explained for low-dimensional Hamiltonian dynamics. 
In this section we briefly explain this story.

Let $(M,g)$ be a closed Riemannian manifold. 
A closed hypersurface $\Sigma \subset M$ is called \textit{minimal} 
if it is a critical point of the volume functional, 
or equivalently if its mean curvature vanishes identically. 
Let $\MH(g)$ denote the set of compact and connected minimal hypersurfaces of $(M,g)$. 

When $\dim M=2$, $\MH(M,g)$ consists of (images of) embedded closed geodesics. 
When $3 \le \dim M \le 7$, the fundamental result is that $\MH(M,g) \ne \emptyset$ for any metric $g$. 
This result follows from the combination of results by Almgren \cite{Almgren}, Pitts \cite{Pitts}, and Schoen-Simon \cite{Schoen_Simon}. 
When $\dim M \ge 8$, for any metric $g$, there exists at least one minimal hypersurface possibly having singularity; see \cite{Zhou} Theorem 1.1.

In recent years, Marques and Neves developed a theory of min-max invariants concerning minimal hypersurfaces. 
Here we discuss it only superficially; see an excellent survey \cite{Marques_JJM} on this theory. 
For any closed Riemannian manifold $(M,g)$, this theory assigns a 
nondecreasing sequence of positive real numbers $(\omega_k(M,g))_{k \ge 1}$ called \textit{volume spectrum}. 
These numbers are min-max values of the volume functional, 
and satisfy the following properties: 

\begin{description} 
\item[(Conformality):] $\omega_k(M,ag) = a^{\dim M - 1} \omega_k(M,g)$ for any $a \in \R_{>0}$.
\item[(Monotonicity):] $\omega_k(M, g') \ge \omega_k(M, g)$ for any metric $g'$ satisfying $g' \ge g$, i.e.  $|v|_{g'} \ge |v|_g \, (\forall v \in TM)$.
\item[(Spectrality):] If $3 \le \dim M \le 7$, each $\omega_k(M,g)$ is an element of 
\[ 
\mathrm{Spec}(M,g):= \bigg\{ \sum_{1 \le j \le N}  m_j \cdot \vol(\Sigma_j) \biggm{|} m_j \in \Z_{>0}, \, \Sigma_j \in \MH(M,g) \bigg\}. 
\] 
\item[(Weyl Law):] For any $d \in \Z_{>0}$, there exists $a(d) \in \R_{>0}$ such that 
\[ 
\lim_{k \to \infty} \frac{\omega_k(M,g)}{k^{1/(d+1)}} = a(d)  \cdot \vol(M,g)^{d/(d+1)}
\] 
for any $d+1$-dimensional closed Riemmanian manifold $(M,g)$. 
\end{description} 

The Weyl law was conjectured by Gromov and proved by Liokumovich-Marques-Neves \cite{LMN}. 
Chodosh-Mantoulidis \cite{Chodosh_Mantoulidis} proved that $a(1)=\sqrt{\pi}$. 
As far as the author knows, $a(d)$ for $d>1$ are not known. 

It is now clear that the volume spectrum behaves quite similarly to ECH/elementary spectral invariants we discussed in Section 5. 
By arguments similar to the proofs of Theorem \ref{thm_strong_closing_dim3}, 
we can prove the following theorem
(Theorem \ref{thm_closing_lemma_for_MH} is implicit in \cite{IMN}): 

\begin{theorem}[\cite{IMN}]\label{thm_closing_lemma_for_MH} 
Let $M$ be a closed manifold with dimension $3 \le \dim M \le 7$, and $g$ be a bumpy metric on $M$. 
For any $h \in C^\infty(M, \R_{\ge 0}) \setminus \{0\}$, there exist $t \in [0,1]$ 
and $\Sigma \in \MH(M, e^{th}g)$ such that $\Sigma \cap \supp(h) \ne \emptyset$. 
\end{theorem} 

A Riemannian metric is called \textit{bumpy} if all compact minimal hypersurfaces are nondegenerate (as critical points of the volume functional). 
It is known that bumpyness is a $C^\infty$-generic condition. The assumption that $g$ is bumpy is required to show that 
$\mathrm{Spec}(M,g)$ is a null set (because it is countable). 

As a corollary of Theorem \ref{thm_closing_lemma_for_MH}, we obtain the generic density of minimal hypersurfaces: 

\begin{corollary}[\cite{IMN}]\label{cor_generic_density_of_MH} 
Let $M$ be a closed manifold with dimension $3 \le \dim M \le 7$. 
Then $\bigcup_{\Sigma \in \MH(M,g)} \Sigma$ is dense in $M$ for a $C^\infty$-generic metric $g$ on $M$. 
\end{corollary} 

Corollary \ref{cor_generic_density_of_MH} 
in particular resolves Yau's conjecture \cite{Yau} for the generic case. 
The conjecture claims that 
any $3$-dimensional closed Riemmanian manifold contains infinitely many (possibly immersed) minimal surfaces.
The conjecture was fully resolved by Song \cite{Song}
who established that 
any closed Riemannian manifold $(M,g)$ with $3 \le \dim M \le 7$ satisfies $\# \MH(M,g) = \infty$. 
Also, 
Marques-Neves-Song \cite{MNS} used the Weyl law to show the existence of equidistributed sequence of minimal hypersurfaces for generic metrics. 
As we explained in Section 6, 
this idea was adopted in Hamiltonian dynamics by \cite{Irie_JSG} and \cite{Prasad_equidistribution}. 

Let us finally discuss the case of surfaces. 
Chodosh-Mantoulidis \cite{Chodosh_Mantoulidis} proved that, 
for any $2$-dimensional closed Riemannian manifold $(M,g)$, 
each $\omega_k(M,g)$ is a finite sum $\sum_j m_j \mathrm{length}(\gamma_j)$, 
where $m_j$ are positive integers and $\gamma_j$ are (possibly immersed) nonconstant closed geodesics. 
This result, combined with continuity and the Weyl law of the volume spectrum, 
produces another proof of Corollary \ref{cor_geodesics}. 
It would be interesting to investigate  relations between the volume spectrum of $2$-dimensional closed Riemannian manifolds
and ECH/elementary capacities of their unit disk cotangent bundles. 

\section*{Acknowledgments.}
I thank 
Rohil Prasad and Dan Cristofaro-Gardiner for answering my question on \cite{CGPZ}, 
Kaoru Ono for answering my question on Gromov-Witten invariants of $\C P^2$, 
and Michael Hutchings for answering my question on elementary capacities.

\end{document}